\documentclass[11]{amsart}
\usepackage{amsfonts,tikz,amsmath,mathtools,amsthm,enumitem,mathrsfs,latexsym,amssymb,blkarray,afterpage,mathrsfs}
\usepackage[margin=1in, paperwidth=8.5in,paperheight=11in]{geometry}
\usepackage{mleftright}
\usepackage{mathtools}
\usepackage{xr-hyper}
\usepackage{hyperref}

\DeclarePairedDelimiter\floor{\lfloor}{\rfloor}

\usepackage{graphicx}

\allowdisplaybreaks

\makeatletter
\@addtoreset{thm}{section}
\makeatother

\makeatletter
\def\@seccntformat#1{\@ifundefined{#1@cntformat}%
    {\csname the#1\endcsname\quad}%      default
    {\csname #1@cntformat\endcsname}}%   individual control
\newcommand{\section@cntformat}{\S\thesection\quad}
\newcommand{\subsection@cntformat}{\S\thesubsection\quad}
\makeatletter

\begin{document}

	\begin{abstract}
	
	Building on MacDonald's formula for the distance from a rank-one projection to the set of nilpotents in $\mathbb{M}_n(\mathbb{C})$, we prove that the distance from a rank $n-1$ projection to the set of nilpotents in $\mathbb{M}_n(\mathbb{C})$ is $\frac{1}{2}\sec\left(\frac{\pi}{\frac{n}{n-1}+2}\right)$. For each $n\geq 2$, we construct examples of pairs $(Q,T)$ where $Q$ is a projection of rank $n-1$ and $T\in\mathbb{M}_n(\mathbb{C})$ is a nilpotent of minimal distance to $Q$. Furthermore, we prove that any two such pairs are unitarily equivalent. We end by discussing possible extensions of these results in the case of projections of intermediate ranks.	
	
	\end{abstract}

\title[The Distance from a Rank $n-1$ Projection to the Nilpotents]{The Distance from a Rank $n-1$ Projection to the Nilpotent Operators on $\mathbb{C}^n$}
\thanks{Research supported in part by NSERC (Canada)}

\author[Z. Cramer]{{Zachary Cramer}}

\newcommand{\Addresses}{{% additional braces for segregating \footnotesize
  \bigskip
  \footnotesize

  Zachary~Cramer, \textsc{Department of Pure Mathematics, University of Waterloo,
    Waterloo, Ontario N2L 3G1}\par\nopagebreak
  \textit{E-mail address}: \texttt{zcramer@uwaterloo.ca}

}}

\date{\today}
\maketitle

%\tableofcontents
%\pagenumbering{arabic}
%\newpage

%\setcounter{page}{1}

\theoremstyle{plain}
\newtheorem{thm}{Theorem}[section]
\newtheorem{lem}[thm]{Lemma}
\newtheorem{cor}[thm]{Corollary}
\newtheorem{prop}[thm]{Proposition}
\newtheorem{conj}[thm]{Conjecture}
\newtheorem{rmk}[thm]{Remark}

\theoremstyle{definition}
\newtheorem{defn}[thm]{Definition}
\newtheorem{exmp}[thm]{Example}
\newtheorem{exer}[thm]{Exercise}

%\makeatletter
%\•addtoreset{thm}{section}
%\makeatother

\newcommand{\cc}{\mathbb{C}}
\newcommand{\rr}{\mathbb{R}}
\newcommand{\dd}{\mathbb{D}}
\newcommand{\cldd}{\overline{\mathbb{D}}}
\newcommand{\epsi}{\varepsilon}
\newcommand{\nn}{\mathbb{N}}
\newcommand{\zz}{\mathbb{Z}}
\newcommand{\fy}{\varphi}
\newcommand{\sign}{\text{sign}}
\newcommand{\bfs}{\textbf{S}}
\newcommand{\triv}{\textbf{1}}
\newcommand{\bb}{\textbf{B}}
\newcommand{\alga}{\mathcal{A}}
\newcommand{\hilb}{\mathcal{H}}
\newcommand{\inv}{\mathrm{GL}}
\newcommand{\tr}{\mathrm{Tr}}
\newcommand{\nil}{\mathrm{Nil}}
\newcommand{\qnil}{\mathrm{QNil}}
\newcommand{\bh}{\mathcal{B(H)}}
\newcommand{\qh}{\mathcal{Q(H)}}
\newcommand{\ol}{\overline}
\newcommand{\dist}{\mathrm{dist}}
\newcommand{\nor}{\mathrm{Nor}}
\newcommand{\mc}{\mathcal}
\newcommand{\mm}{\mathbb{M}}
\newcommand{\au}{\sim_{au}}
\newcommand{\sorb}{\mathcal{S}}
\newcommand{\alg}{\mathrm{Alg}}
\newcommand{\bqt}{\mathrm{BQT}}
\newcommand\scalemath[2]{\scalebox{#1}{\mbox{\ensuremath{\displaystyle #2}}}}
\newcommand*\hedgehog{\includegraphics[width=1.2em]{Images/hedgehog4.png}} 
\renewcommand{\qedsymbol}{$\blacksquare$}
	
	\section{Introduction}
	
Let $\mc{H}$ be a complex Hilbert space of (possibly infinite) dimension $n$, and let $\mc{B(H)}$ denote the algebra of bounded linear operators acting on $\mc{H}$. Consider the sets
$$\begin{array}{l}\mc{P(H)}=\left\{P\in\mc{B(H)}:P=P^2=P^*\right\}\setminus\{0\},\,\,\text{and}\vspace{0.2cm}\\
\mc{N(H)}=\left\{N\in\mc{B(H)}:N^j=0\,\,\,\text{for some}\,\,\,j\in\nn\right\}\end{array}\smallskip$$ consisting of all non-zero orthogonal projections on $\mc{H}$ and all nilpotent operators on $\mc{H}$, respectively. We are interested in the problem of understanding the distance between these two sets, measured in the usual operator norm on $\mc{B(H)}$. This quantity will be denoted by $\delta_n$:
$$\delta_{n}\coloneqq\mathrm{dist}(\mc{P(H)},\mc{N}(\mc{H}))=\inf\left\{\|P-N\|:P\in\mc{P(H)},N\in\mc{N(H)}\right\}.$$

The problem of computing $\delta_n$ is by no means new to the world of operator theory. In 1972, Hedlund~\cite{Hedlund} proved that $\delta_2=1/\sqrt{2}$, and that $1/4\leq\delta_{n}\leq 1$ for all $n\geq 3$. This lower bound was increased to $1/2$ by Herrero~\cite{Her1} shortly thereafter. At this time Herrero also showed that $\delta_n=1/2$ whenever $n$ is infinite, thus reducing the problem to the case in which $\mc{H}=\cc^n$ for some $n\in\nn$, $n\geq 3$.

Various estimates on the values of $\delta_n$ were obtained in the early 1980's. One such estimate established by Salinas \cite{SalinasDistance} states that $$\begin{array}{ccc}\phantom{\text{for all}\,\,\,n\in\nn,} & \displaystyle{\frac{1}{2}\leq\delta_n\leq\frac{1}{2}+\frac{1+\sqrt{n-1}}{2n}} & \text{for all}\,\,\,n\in\nn.\end{array}$$ One may note that this upper bound approaches $1/2$ as $n$ tends to infinity, and hence Salinas' inequality leads to an alternative proof that $\delta_{\aleph_0}=1/2$. Herrero \cite{HerreroUnitaryOrbitsofPower} subsequently improved upon this upper bound for large values of $n$ by showing that $$\begin{array}{ccc}
\phantom{\text{for}\,\,\,n\geq 2,} & \displaystyle{\frac{1}{2}\leq\delta_n\leq \frac{1}{2}+\sin\left(\frac{\pi}{\floor{\frac{n+1}{2}}}\right)} & \text{for}\,\,\,n\geq 2,\end{array}$$ where $\floor{\cdot}$ denotes the greatest integer function. 

For many years the bounds obtained by Salinas and Herrero remained the best known. In 1995, however, MacDonald~\cite{MacDonaldProjections} established a new upper bound that would  improve upon these estimates for all values of $n$. In order to describe MacDonald's approach, we first make the following remarks.
\begin{itemize}
\item[(i)] Any two projections in $\mm_n(\cc)$ of equal rank are unitarily equivalent, and thus of equal distance to $\mc{N}(\cc^n)$. As a result, $\displaystyle{\delta_n=\displaystyle{\min_{1\leq r\leq n}\nu_{r,n}},}$ where $$\nu_{r,n}\coloneqq\inf\left\{\|P-N\|:P\in\mc{P}(\cc^n),\,rank(P)=r,N\in\mc{N}(\cc^n)\right\}.\smallskip$$

\item[(ii)] Straightforward estimates show that when computing $\nu_{r,n}$, one need only consider nilpotents of norm at most $2$. From here, one may use the compactness of the set of projections in $\mm_n(\cc)$ of rank $r$ and the set of nilpotents in $\mm_n(\cc)$ of norm at most $2$ to show that $\nu_{r,n}$ is achieved by some projection-nilpotent pair, and hence so too is $\delta_n$.\\

\item[(iii)]If $\{e_i\}_{i=1}^n$ denotes the standard basis for $\cc^n$, then 
$$\nu_{r,n}=\min\left\{\|P-N\|:P\in\mc{P}(\cc^n),\,rank(P)=r,N\in\mc{T}_n\right\}$$ where $\mc{T}_n$ is the algebra of operators that are strictly upper triangular as matrices written with respect to $\{e_i\}_{i=1}^n$.\\ 
\end{itemize}

The reduction from $\mc{N}(\cc^n)$ to $\mc{T}_n$ described in (iii) may seem innocuous at first glance. This alternate formulation, however, allows one to make use of a theorem of Arveson~\cite{Arveson} that describes the distance from an operator in $\mc{B(H)}$ to a nest algebra. The version of this result that we require was established by Power~\cite{PowerDistance}, and is presented below for the algebra $\mc{T}_n$. Note that for vectors $x,y\in\cc^n$, the notation $x\otimes y^*$ is used to denote the rank-one operator $z\mapsto \langle z,y\rangle x$ acting on $\cc^n$.

\begin{thm}[Arveson Distance Formula]\label{Arveson distance formula}
Let $\{e_i\}_{i=1}^n$ denote the standard basis for $\cc^n$. Define $E_0=0$ and $E_k=\sum_{i=1}^ke_i\otimes e_i^*$ for each $k\in\{1,2,\ldots, n\}$. For any $A\in\mm_n(\cc)$, $$\displaystyle{\dist(A,\mc{T}_n)=\max_{1\leq i\leq n}\|E_{i-1}^\perp AE_i\|.}\smallskip$$

\end{thm}

\noindent Using Arveson's formula, MacDonald successfully determined the exact value of $\nu_{1,n}$, the distance from a rank-one projection in $\mm_n(\cc)$ to $\mathcal{N}(\cc^n)$.
	\begin{thm}\cite[Theorem 1]{MacDonaldProjections}\label{MacDonald distance formula}
	For every positive integer $n$, the distance from the set of rank-one projections in $\mm_n(\cc)$ to $\mc{N}(\cc^n)$ is $$\nu_{1,n}=\frac{1}{2}\sec\left(\frac{\pi}{n+2}\right).\smallskip$$ 	
	\end{thm}

	The expression for $\nu_{1,n}$ described above provides an upper bound on $\delta_n$ that is sharper than those previously obtained by Herrero and Salinas for all $n\in\nn$. In addition, MacDonald proved that this bound is in fact optimal when $n=3$ \cite[Corollary 4]{MacDonaldProjections}. These results led to the formulation of the following conjecture.
\begin{conj}[MacDonald, \cite{MacDonaldProjections}]\label{MacDonald's Conjecture}
The closest non-zero projections to $\mc{N}(\cc^n)$ are of rank $1$. That is,  
$$\delta_n=\nu_{1,n}=\frac{1}{2}\sec\left(\frac{\pi}{n+2}\right)\,\,\,\text{for all}\,\,\,n\in\nn.\smallskip$$ 

\end{conj}

\noindent Conjecture~\ref{MacDonald's Conjecture} has since been verified for $n=4$ \cite[Theorem 3.4]{MacDonaldIdempotents}, but remains open for all $n\geq 5$.

MacDonald's success in computing $\nu_{1,n}$ relied heavily on the simple structure of rank-one projections in $\mm_n(\cc)$. Specifically, the fact that every such projection can be written as a simple tensor $P=e\otimes e^*$ for some unit vector $e\in\cc^n$ made it feasible to obtain closed-form expressions for the norms $\|E_{i-1}^\perp PE_i\|$ in terms of the entries of $P$. MacDonald was then able to prove using the Arveson Distance Formula that the rank-one projections of minimal distance to $\mc{T}_n$ are such that $\|E_{i-1}^\perp PE_i\|=\nu_{1,n}$ for all $i\in\{1,2,\ldots,n\}$. An exact expression for $\nu_{1,n}$ was then derived through algebraic and combinatorial arguments.

Extending this approach to accommodate projections of intermediate ranks appears to be a formidable task; when $P$ is not expressible as a simple tensor $e\otimes e^*$ it becomes significantly more challenging to obtain explicit formulas for $\|E_{i-1}^\perp PE_i\|$. One may note, however, that the simple structure that led to success in the rank-one case can similarly be observed in projections $I-e\otimes e^*$ of rank $n-1$. It is therefore the goal of this paper to extend MacDonald's approach to determine the exact value of $\nu_{n-1,n}$.

We accomplish this goal in three stages. Motivated by the analogous result for projections of rank $1$, we show in \S2 that any projection $Q$ of rank $n-1$ that is of minimal distance to $\mc{T}_n$ must be such that $\|E_{i-1}^\perp QE_i\|=\nu_{n-1,n}$ for all $i$. In \S3, we then apply these equations to determine a list of candidates for $\nu_{n-1,n}$ via arguments adapted from \cite{MacDonaldProjections}. Finally, we prove that exactly one such candidate satisfies a certain necessary norm inequality from \cite{MacDonaldIdempotents}, and therefore deduce that  this value must be $\nu_{n-1,n}$.

In \S4 we outline a construction of the pairs $(Q,T)$ where $Q\in\mm_n(\cc)$ is a projection of rank $n-1$, $T$ is an element of $\mc{T}_n$, and $\|Q-T\|=\nu_{n-1,n}$. We prove that for each $n\in\nn$, any two such pairs are, in fact, unitary equivalent. Lastly, in \S5 we propose a possible formula for $\nu_{r,n}$ in the case of projections of arbitrary rank, which can be seen to closely resemble numerical estimates for $\nu_{r,n}$ when $n$ is small. We briefly explain how this formula could be used to answer MacDonald's conjecture in the affirmative.

	\section{Equality in Arveson's Distance Formula}\label{sec2}
	
		Fix an integer $n\geq 3$, and let $\{e_1,e_2,\ldots, e_n\}$ denote the standard basis for $\cc^n$. Define $E_0=0$ and  $E_k=\sum_{i=1}^ke_i\otimes e_i^*$ for each $k\in\{1,2,\ldots, n\}$. Throughout, $Q=(q_{ij})$ will denote a projection in $\mm_n(\cc)$ of rank $n-1$ that is of minimal distance to $\mc{T}_n$. Additionally, let $e\in\cc^n$ be a unit vector such that $Q=I-e\otimes e^*$, and let $P=(p_{ij})$ denote the rank-one projection $e\otimes e^*$. 
			
		The goal of this section is to derive a sequence of equations relating the entries of $Q$ to the distance $\nu_{n-1,n}$. Our strategy will be to use the algebraic relations satisfied by the entries of $Q$ to derive closed-form expressions for the norms $\|E_{i-1}^\perp QE_i\|$. Next, we will relate these expression to $\nu_{n-1,n}$ through the Arveson Distance Formula. 

In the case of rank-one projections, MacDonald obtained closed-form expressions for the norms in the Arveson Distance Formula by analysing the sequence of partial traces associated to such a projection. Specifically, this sequence $\{a_i\}_{i=0}^n$ is defined by setting $a_0=0$ and 
	\begin{equation}\label{defining ak}
	\begin{array}{rl}	
	\displaystyle{a_k=\sum_{i=1}^kp_{ii}=k-\sum_{i=1}^kq_{ii},} & k\in\{1,2,\ldots, n\}.
	\end{array}\smallskip
	\end{equation}
We may then express the entries of $e$ in terms of $\{a_i\}_{i=0}^n$ as 
$$e=\begin{bmatrix}\,z_1\sqrt{a_1-a_0} & z_2\sqrt{a_2-a_1} & \cdots & z_n\sqrt{a_n-a_{n-1}}\,\end{bmatrix}^T,$$
where $z_1,z_2,\ldots, z_n$ are complex numbers of modulus $1$. 

\begin{rmk}\label{simplifying the projections}\textup{By defining $U\in\mm_n(\cc)$ to be the diagonal unitary $U=\mathrm{diag}(z_1,z_2,\ldots, z_n)$ and replacing $Q$ with the unitarily equivalent projection $U^*QU$, we may assume that each of the complex numbers $z_i$ is equal to $1$. That is, we may assume that $q_{ij}\leq 0$ for all distinct indices $i$ and $j$. Since $U$ commutes with each of the projections $E_i$,  the norms $\|E_{i-1}^\perp QE_i\|$---and hence $\mathrm{dist}(Q,\mc{T}_n)$---are preserved by this transformation.}
\end{rmk}

Under the assumption of Remark~\ref{simplifying the projections}, one readily obtains useful identities among the entries of $P$ and $Q$. Notably, the entries on the off-diagonals of these projections can by described entirely be those on the diagonals:
\smallskip
\begin{equation}\label{projection from sequence}
\begin{array}{cccc} p_{ij}=\sqrt{p_{ii}p_{jj}} & \text{and} & q_{ij}=-\sqrt{(1-q_{ii})(1-q_{jj})} & \text{for all}\,\,i\neq j.\end{array}\smallskip
\end{equation}
One may then verify that\smallskip
	\begin{equation}\label{projection structure equation}
		\begin{array}{ccccc}
			p_{ij}p_{ik}=p_{ii}p_{jk} & \text{and} & q_{ij}q_{ik}=-q_{jk}(1-q_{ii}) & \text{for all}\,\,i,j,k\,\,\text{distinct}.
		\end{array}\smallskip
		\end{equation}
These identities will be used heavily throughout the proof of the following lemma, which serves as the first step toward obtaining closed-form descriptions of the norms $\|E_{i-1}^\perp QE_i\|$ in terms of the sequence $\{a_i\}_{i=0}^n$.

	\begin{lem}\label{Finding the entries and eigenvalues of Q_k^*Q_k}
	
		Let $Q=(q_{ij})$ be a projection in $\mm_n(\cc)$ of rank $n-1$, and let $\left\{a_i\right\}_{i=0}^n$ denote the non-decreasing sequence from equation~(\ref{defining ak}). For $k\in\{1,2,\ldots, n\}$, define $Q_k\coloneqq E_{k-1}^\perp QE_k$, and let $B_k$ denote the restriction of $Q_k^*Q_k$ to the range of $E_k$.\smallskip	
		
		\begin{itemize}
		
			\item[(i)]If $q_{ij}\leq 0$ for all $i\neq j$, then the entries of $B_k=(b_{ij})$ are given by
			$$
			b_{ij}=\left\{\begin{array}{ll}
			\phantom{-}q_{kk}-a_{k-1}(1-q_{kk}) & \text{if}\,\,i=j=k,\vspace{0.1cm}\\
			\phantom{-}(1-a_{k-1})(1-q_{ii}) & \text{if}\,\,i=j\neq k, \vspace{0.1cm}\\
			-(1-a_{k-1})q_{ij} & \text{if}\,\,i,j,k\,\,\text{are distinct},\vspace{0.1cm}\\
			\phantom{-}a_{k-1}q_{ij} & \text{otherwise.}
			\end{array}\right.
			$$
			
			\item[(ii)]
			We have
		$$\|Q_k\|^2=\frac{\mathrm{Tr}(B_k)+\sqrt{2\mathrm{Tr}(B_k^2)-\mathrm{Tr}(B_k)^2}}{2}.$$
		\end{itemize}
	\end{lem}
	
	\begin{proof}
	
		First, suppose that $q_{ij}\leq 0$ for all $i\neq j$. Since $Q$ is idempotent, its entries $q_{ij}$ satisfy the equation 
		$
		q_{ij}=\sum_{\ell=1}^nq_{i\ell}q_{\ell j}.$
	This equation, together with the identities from (\ref{projection structure equation}), allows one to compute the entries of $B_k$ directly. Indeed, 
	
	$$\begin{array}{rclcl}
	b_{kk}&=&\displaystyle{q_{kk}^2+q_{k+1,k}^2+\cdots+q_{nk}^2}\vspace{0.1cm}\\
	&=&\displaystyle{q_{kk}-q_{1k}^2-q_{2k}^2-\cdots-q_{k-1,k}^2}\vspace{0.1cm}\\
	&=&\displaystyle{q_{kk}-\sum_{\ell=1}^{k-1}(1-q_{\ell\ell})(1-q_{kk})}\hspace{0.2cm} = \hspace{0.2cm} \displaystyle{q_{kk}-a_{k-1}(1-q_{kk})},
	\end{array}$$
	
	\noindent and if $i\neq k$, then
	$$\begin{array}{rclcl}
	b_{ii}&=&\displaystyle{q_{ki}^2+q_{k+1,i}^2+\cdots+q_{ni}^2}\vspace{0.1cm}\\
	&=&\displaystyle{q_{ii}-q_{1i}^2-q_{2i}^2-\cdots-q_{k-1,i}^2}\vspace{0.1cm}\\
	&=&\displaystyle{q_{ii}-q_{ii}^2-\sum_{\ell=1,\ell\neq i}^{k-1}(1-q_{\ell\ell})(1-q_{ii})}\vspace{0.1cm}\\
	&=&\displaystyle{(1-q_{ii})\left((k-2)-\sum_{\ell=1}^{k-1}q_{\ell\ell}\right)}\hspace{0.2cm} = \hspace{0.2cm} \displaystyle{(1-a_{k-1})(1-q_{ii})}.
	\end{array}$$
	
	\noindent If $i,j,$ and $k$ are all distinct, then	
	$$\begin{array}{rclcl}
	b_{ij}&=&\displaystyle{q_{ki}q_{kj}+q_{k+1,i}q_{k+1,j}+\cdots+q_{ni}q_{nj}}\vspace{0.1cm}\\
	&=&\displaystyle{q_{ij}-q_{1i}q_{1j}-q_{2i}q_{2j}-\cdots-q_{k-1,i}q_{k-1,j}}\vspace{0.1cm}\\
	&=&\displaystyle{q_{ij}-q_{ii}q_{ij}-q_{ji}q_{jj}+\sum_{\ell=1,\ell\neq i,j}^{k-1}q_{ij}(1-q_{\ell\ell})}\vspace{0.1cm}\\
	&=&\displaystyle{q_{ij}\left((k-2)-\sum_{\ell=1}^{k-1}q_{\ell\ell}\right)}\hspace{0.2cm}=\hspace{0.2cm}\displaystyle{-(1-a_{k-1})q_{ij}.}
	\end{array}$$
	
	\noindent Lastly, we consider entries $b_{ij}$ for which $i<j=k$ or $j<i=k$. Since $B_k=B_k^*$, it suffices to establish the formula for $b_{ij}$ in the case that $i<j=k$. We have
	$$\begin{array}{rclcl}
	b_{ik}&=&\displaystyle{q_{ki}q_{kk}+q_{k+1,i}q_{k+1,k}+\cdots+q_{ni}q_{nk}}\vspace{0.1cm}\\
	&=&\displaystyle{q_{ik}-q_{1i}q_{1k}-q_{2i}q_{2k}-\cdots-q_{k-1,i}q_{k-1,k}}\vspace{0.1cm}\\
	&=&\displaystyle{q_{ik}-q_{ii}q_{ik}+\sum_{\ell=1,\ell\neq i}^{k-1}q_{ik}(1-q_{\ell\ell})}\vspace{0.1cm}\\
	&=&\displaystyle{q_{ik}\left((k-1)-\sum_{\ell=1}^{k-1}q_{\ell\ell}\right)}\hspace{0.2cm} = \hspace{0.2cm} \displaystyle{a_{k-1}q_{ik}}.
	\end{array}\vspace{0.1cm}$$
	
	We now turn our attention to the proof of (ii). Note that if $P$ denotes the rank-one projection $I-Q$, then  $$Q_k=E_{k-1}^\perp(I-P)E_k=e_{k}\otimes e_k^*-E_{k-1}^\perp PE_k.$$ 
	Thus, $Q_k$---and hence $B_k$---has rank at most $2$. It follows that $B_k$ has at most two non-zero eigenvalues $\lambda_0$ and $\lambda_1$, which can be obtained by solving the system of equations
	$$
	\left\{\begin{array}{lcl}
	\lambda_0+\lambda_1&=&\mathrm{Tr}(B_k), \vspace{0.1cm}\\
	\lambda_0^2+\lambda_1^2&=&\mathrm{Tr}(B_k^2).
	\end{array}\right.
	$$
	The solutions to this system are $\lambda=\frac{1}{2}\left(\tr(B_k)\pm \sqrt{2\tr(B_k^2)-\tr(B_k)^2}\right)$, and therefore the result follows.
	\end{proof}
	\smallskip
	
	\begin{thm}\label{function definition of Arveson seminorm thm}
	
		Let $Q=(q_{ij})$ be a projection in $\mm_n(\cc)$ of rank $n-1$, and let $\left\{a_i\right\}_{i=0}^n$ denote the non-decreasing sequence from equation~(\ref{defining ak}). If $f:[0,1]\times[0,1]\rightarrow\rr$ denotes the function 
$$
		f(x,y)=\frac{\sqrt{x^2y^2-4x^2y+2xy^2+4x^2-2xy+y^2-2y+1}-xy-y+2x+1}{2},\smallskip
$$
		
		\noindent then for each $k\in\{1,2,\ldots, n\}$, $\|E_{k-1}^\perp QE_k\|^2=f(a_{k-1},a_k)$.
		
	\end{thm}

	\begin{proof}
	
	By Remark~\ref{simplifying the projections}, we may conjugate $Q$ by a diagonal unitary if necessary and assume that $q_{ij}\leq 0$ for all $i\neq j$. Fix an integer $k\in\{1,2,\ldots, n\}$, define $Q_k\coloneqq E_{k-1}^\perp QE_k$, and let $B_k=(b_{ij})$ denote the restriction of $Q_k^*Q_k$ to the range of $E_k$. By Lemma~\ref{Finding the entries and eigenvalues of Q_k^*Q_k}~(ii),
	$$\|Q_k\|^2=\frac{\mathrm{Tr}(B_k)+\sqrt{2\mathrm{Tr}(B_k^2)-\mathrm{Tr}(B_k)^2}}{2}.\smallskip$$ 
	
	Using the expressions for the entries of $B_k$ derived in Lemma~\ref{Finding the entries and eigenvalues of Q_k^*Q_k} (i), we find that  
	$$
	\begin{array}{rcl}
	\mathrm{Tr}(B_k)&=&\displaystyle{\sum_{i=1}^{k-1}b_{ii}}+b_{kk}\vspace{0.1cm}\\
	&=&\displaystyle{\sum_{i=1}^{k-1}(1-a_{k-1})(1-q_{ii})+q_{kk}-a_{k-1}(1-q_{kk})}\vspace{0.1cm}\\
	&=&a_{k-1}(1-a_{k-1})+q_{kk}-a_{k-1}(1-q_{kk})\vspace{0.1cm}\\
	&=&q_{kk}+a_{k-1}(q_{kk}-a_{k-1})\vspace{0.1cm}\\
	&=&q_{kk}+a_{k-1}(1-a_k).\end{array}$$
	Moreover, if $B_k^2=(c_{ij})$, then
	$$\begin{array}{rcl}
	c_{kk}&=&\displaystyle{b_{kk}^2+\sum_{\ell=1}^{k-1}b_{k\ell}^2}\vspace{0.1cm}\\
	&=&\displaystyle{\left(q_{kk}-a_{k-1}(1-q_{kk})\right)^2+\sum_{\ell=1}^{k-1}a_{k-1}^2q_{k\ell}^2}\vspace{0.1cm}\\
	&=&\displaystyle{\left(q_{kk}-a_{k-1}(1-q_{kk})\right)^2+\sum_{\ell=1}^{k-1}a_{k-1}^2(1-q_{kk})(1-q_{\ell\ell})}\vspace{0.1cm}\\
	&=&\displaystyle{\left(q_{kk}-a_{k-1}(1-q_{kk})\right)^2+a_{k-1}^3(1-q_{kk}),}\end{array}
$$
	and for $i\leq k-1$,
	$$\begin{array}{rcl}
	c_{ii}&=&\displaystyle{b_{ii}^2+b_{ik}^2+\sum_{\ell=1,\ell\neq i}^{k-1}b_{i\ell}^2}\vspace{0.1cm}\\
	&=&\displaystyle{(1-a_{k-1})^2(1-q_{ii})^2+a_{k-1}^2q_{ik}^2+\sum_{\ell=1,\ell\neq i}^{k-1}(1-a_{k-1})^2q_{i\ell}^2}\vspace{0.1cm}\\
	&=&\displaystyle{(1-a_{k-1})^2(1-q_{ii})^2+a_{k-1}^2(1-q_{ii})(1-q_{kk})+\sum_{\ell=1,\ell\neq i}^{k-1}(1-a_{k-1})^2(1-q_{ii})(1-q_{\ell\ell})}\vspace{0.1cm}\\
	&=&\displaystyle{a_{k-1}(1-q_{ii})\left((1-a_{k-1})^2+a_{k-1}(1-q_{kk})\right)}.\end{array}$$
	Thus, 
	$$\begin{array}{rcl}
	\mathrm{Tr}(B_k^2)&=&\displaystyle{c_{kk}+\sum_{\ell=1}^{k-1}a_{k-1}(1-q_{\ell\ell})\left((1-a_{k-1})^2+a_{k-1}(1-q_{kk})\right)}\vspace{0.1cm}\\
	&=&\displaystyle{\left(q_{kk}-a_{k-1}(1-q_{kk})\right)^2+a_{k-1}^3(1-q_{kk})+a_{k-1}^2\left((1-a_{k-1})^2+a_{k-1}(1-q_{kk})\right).}
\end{array}\vspace{0.2cm}$$\vspace{0.1cm}

\vspace{-0.2cm} \noindent These descriptions of $\mathrm{Tr}(B_k)$ and $\mathrm{Tr}(B_k^2)$ allow one to express $\|Q_k\|^2$ as a function of $a_{k-1}$, $a_k$, and $q_{kk}$. The desired formula for $\|Q_k\|^2$ may now be obtained by writing ${q_{kk}=1-(a_k-a_{k-1})}$. 
	\end{proof}
	\smallskip
	
	Our first goal of this section is now complete: we have derived a closed-form expression for each norm $\|E_{i-1}^\perp QE_i\|$ in terms of the sequence $\{a_i\}_{i=0}^n$. In order to show that every such norm is equal to $\nu_{n-1,n}$, we must first investigate the properties of the function $f$ from Theorem~\ref{function definition of Arveson seminorm thm}.
	
%	From its definition it is easy to see that $\{a_i\}_{i=0}^n$ increases monotonically from $a_0=0$ to $a_n=\mathrm{Tr}(P)=1$. Moreover, the above description of $e$ in terms of the sequence $\{a_i\}_{i=0}^n$ demonstrates that any sequence increasing monotonically from $0$ to $1$ can arise as the sequence of partial traces of some rank-one projection. We record this fact below for future reference.
%	
%	\begin{lem}\label{every sequence increasing from 0 to 1 works lemma}
%		If $\{a_i\}_{i=0}^n$ is a sequence that increases monotonically from $a_0=0$ to $a_n=1$, then there is a rank-one projection $R=(r_{ij})$ in $\mm_n(\cc)$ such that $r_{ij}\geq 0$ for all $i$ and $j$, and $a_k=\sum_{i=1}^kr_{ii}$ for each $k\in\{1,2,\ldots, n\}$. \smallskip
%	\end{lem}	
	
	\begin{lem}\label{f increasing in x, decreasing in y lemma}
	
	If $f:[0,1]\times[0,1]\rightarrow\rr$ denotes the function
		$$f(x,y)=\frac{\sqrt{x^2y^2-4x^2y+2xy^2+4x^2-2xy+y^2-2y+1}-xy-y+2x+1}{2},\smallskip$$ 
		
		\noindent then $f$ is increasing in $x$ and decreasing in $y$. Moreover, $0\leq f(x,y)\leq 1$ whenever $0\leq x\leq y\leq 1$.
	
	\end{lem}
	
	\begin{proof}
	
		Define $g:[0,1]\times[0,1]\rightarrow \rr$ by 
		\begin{align*}
		g(x,y)&=x^2y^2-4x^2y+2xy^2+4x^2-2xy+y^2-2y+1,
		\end{align*}
		so $f(x,y)=\frac{1}{2}(\sqrt{g(x,y)}-xy-y+2x+1)$. We begin by proving that $g(x,y)$ is non-negative on its domain and zero only at $(0,1)$. We will therefore verify that $f$ is well-defined on $[0,1]\times[0,1]$, and that the partial derivatives of $f$ exist at all points $(x,y)\neq(0,1)$. 
		
		Observe that for each fixed $y\in[0,1]$, the map 
		$$x\mapsto g(x,y)=(2-y)^2x^2-2y(1-y)x+(1-y)^2\smallskip$$
		defines a convex quadratic on $[0,1]$ with vertex at $x_0=y(1-y)/(2-y)^2$. If $y\in[0,1)$, then 
		$$g(x_0,y)=g\left(\frac{y(1-y)}{(2-y)^2},y\right)=\frac{4(1-y)^3}{(2-y)^2}>0.\smallskip$$
		Consequently, $g(x,y)>0$ for all $(x,y)\in[0,1]\times[0,1).$ Note as well that at $y=1$ we have $g(x,1)=x^2$. It follows that $g(0,1)=0$ and $g(x,y)>0$ for all other values of $(x,y)\in[0,1]\times[0,1]$. Thus, $f$ is well-defined, and the partial derivatives \vspace{0.1cm}
%%%		$$\begin{array}{ccc}
%%% 		f_x(x,y)=\displaystyle{\frac{g_x(x,y)}{4\sqrt{g(x,y)}}-\frac{y+2}{2}} &  \hspace{0.5cm} \vspace{0.5cm}	& \displaystyle{f_{xx}(x,y)=\frac{2g(x,y)g_{xx}(x,y)-\left(g_x(x,y)\right)^2}{8\left(g(x,y)\right)^{3/2}}}=\displaystyle{\frac{2(1-y)^3}{\left(g(x,y)\right)^{3/2}}}\\
%%%		\displaystyle{f_y(x,y)=\frac{g_y(x,y)}{4\sqrt{g(x,y)}}-\frac{x+1}{2}} & \hspace{0.5cm} & \displaystyle{f_{yy}(x,y)=\frac{2g(x,y)g_{yy}(x,y)-\left(g_y(x,y)\right)^2}{8\left(g(x,y)\right)^{3/2}}}=\displaystyle{\frac{2x^3}{\left(g(x,y)\right)^{3/2}}}
%%%		\end{array}$$
%%%		exist for all $(x,y)\neq(0,1)$.
		$$\begin{array}{lcl}
 		f_x(x,y)=\displaystyle{\frac{g_x(x,y)}{4\sqrt{g(x,y)}}+\frac{2-y}{2}},&  \hspace{0.5cm} \vspace{0.5cm}	& \displaystyle{f_{xx}(x,y)=\frac{2(1-y)^3}{\left(g(x,y)\right)^{3/2}}},\\
		\displaystyle{f_y(x,y)=\frac{g_y(x,y)}{4\sqrt{g(x,y)}}-\frac{x+1}{2}}, & \hspace{0.5cm} & \displaystyle{f_{yy}(x,y)=\frac{2x^3}{\left(g(x,y)\right)^{3/2}}}\\
		\end{array}\smallskip$$
		exist for all $(x,y)\neq(0,1)$.
		
		Our next task is to prove that $f(x,y)$ is increasing in $x$. First observe that $f(x,1)=x$ is clearly increasing. Furthermore, for every fixed $y\in[0,1)$, $f_{xx}(x,y)$ is well-defined and strictly positive for all $x$. Hence, $$f_x(x,y)=\displaystyle{\frac{xy^2-4xy+y^2+4x-y}{2\sqrt{g(x,y)}}+\frac{2-y}{2}}\smallskip$$ is an increasing function of $x$. We conclude that $f_x(x,y)\geq f_x(0,y)=1-y>0$ for every $x\in[0,1]$. Thus, $f$ is an increasing function of $x$ on $[0,1]$.
	
	We now use a similar argument to show that $f$ is a decreasing function of $y$. For $x=0$, we have that $f(0,y)=1-y$ is clearly decreasing. Now given a fixed $x\in(0,1]$, it is clear from above that $f_{yy}(x,y)$ is well-defined and strictly positive for all $y$. It follows that the partial derivative
	$$f_y(x,y)=\displaystyle{\frac{x^2y-2x^2+2xy-x+y-1}{2\sqrt{g(x,y)}}-\frac{x+1}{2}}\smallskip$$
	
	\noindent is an increasing function of $y$ on $[0,1]$. Hence
	$f_y(x,y)\leq f_y(x,1)=-x<0$ for every $y\in[0,1]$.
	This proves that $f$ is a decreasing function of $y$ on $[0,1]$, as desired.
	
	For the final claim suppose that $0\leq x\leq y\leq 1$. Consider the sequence $\{a_i\}_{i=0}^3$ defined by $a_0=0$, $a_1=x$, $a_2=y$, and $a_3=1$, as well as the vector $e=\begin{bmatrix}\sqrt{x} & \sqrt{y-x} & \sqrt{y}\end{bmatrix}^T.$ Note that $Q\coloneqq I-e\otimes e^*$ is a rank-two projection in $\mm_3(\cc)$ with partial trace sequence given by $\{a_i\}_{i=0}^3$. It then follows from Theorem~\ref{function definition of Arveson seminorm thm} that $$f(x,y)=f(a_1,a_2)=\|E_1^\perp QE_2\|^2,$$ and hence $0\leq f(x,y)\leq 1$.
%		It can be shown that, in fact,
%		$$\begin{array}{ccc}
%		f_{xx}(x,y)=\displaystyle{\frac{4(1-y)^3}{\left(g(x,y)\right)^{3/2}}} & \text{and} & f_{yy}(x,y)=\displaystyle{\frac{4x^3}{\left(g(x,y)\right)^{3/2}}},
%		\end{array}$$
%		hence $f_{xx}(x,y)>0$ and $f_{yy}(x,y)>0$ for all $x\neq 0$ and all $y\neq 0$. This means that $f_x:(0,1]\times[0,1)\rightarrow \rr$ is increasing
	\end{proof}
	\smallskip
	
%%%	\begin{lem}
%%%	
%%%		Let $Q=(q_{ij})$ be a rank $n-1$ projection in $\mm_n$ such that $q_{ij}\leq 0$ for all $i\neq j$.  Given any index $k\in\{1,2,\ldots, n-1\}$, there is a rank $n-1$ projection $Q^\prime=(q_{ij}^\prime)$ in $\mm_n$ such that $q_{ij}^\prime\leq 0$ for all $i\neq j$; $\|Q_i^\prime\|=\|Q_i\|$ for all $i\neq k,k+1$; and $\|Q_k^\prime\|=\|Q_{k+1}^\prime\|\leq\max\left\{\|Q_k\|,\|Q_{k+1}\|\right\}$.
%%%	
%%%	\end{lem}

	\begin{thm}\label{Arveson seminorms must be equal theorem}
	
		If $Q\in\mm_n(\cc)$ is a projection of rank $n-1$ that is of minimal distance to $\mc{T}_n$, then $\|E_{i-1}^\perp QE_i\|=\|E_{j-1}^\perp QE_j\|$ for all $i$ and $j$.
		
	\end{thm}
	
	\begin{proof}
	
	By Remark~\ref{simplifying the projections}, we may assume without loss of generality that $Q=(q_{ij})$ where $q_{ij}\leq 0$ for all  $i\neq j$. Let $\{a_i\}_{i=0}^n$ denote the non-decreasing sequence from equation~(\ref{defining ak}), and for each $i\in\{1,2,\ldots, n\}$, define $Q_i\coloneqq E_{i-1}^\perp QE_i$. Suppose to the contrary that not all values of $\|Q_i\|$ are equal. Define 
	$$\mu\coloneqq\max_{1\leq i\leq n}\|Q_i\|,$$ and let $j$ denote the largest index in $\{1,2,\ldots, n\}$ such that $\|Q_j\|=\mu$.
	
	First consider the case in which $j=n$. Let $k$ denote the largest index in $\{1,2,\ldots, n-1\}$ such that $\|Q_k\|<\mu$. With $f$ as in Theorem~\ref{function definition of Arveson seminorm thm}, we have that $$f(a_{k-1},a_k)=\|Q_k\|^2<\|Q_{k+1}\|^2=f(a_{k},a_{k+1}).$$
	Thus, if $g:\left[a_{k-1},a_{k}\right]\rightarrow \rr$ is given by
	$$g(x)=f(a_{k-1},x)-f(x,a_{k+1}),$$ then 
$g(a_{k})=f(a_{k-1},a_{k})-f(a_{k},a_{k+1})< 0,$ while
$g(a_{k-1})=1-f(a_{k-1},a_{k+1})\geq 0$ by Lemma~\ref{f increasing in x, decreasing in y lemma}. Since $g$ is continuous on its domain, the Intermediate Value Theorem  gives rise to some $a_k^\prime\in\left[a_{k-1},a_{k}\right]$ such that $g(a_k^\prime)=0$. By replacing $a_k$ with $a_k^\prime$ in the sequence $\{a_i\}_{i=0}^n$, one may equate $\|Q_k\|$ and $\|Q_{k+1}\|$ while leaving the remaining norms $\|Q_i\|$ unchanged. Most importantly, since $a_k^\prime\leq a_k$, Lemma~\ref{f increasing in x, decreasing in y lemma} implies that the new common value of $\|Q_k\|$ and $\|Q_{k+1}\|$ is strictly less than $\mu$. 

This argument may now be repeated to successively reduce the norms $\|Q_i\|$ for $i>k$ to values strictly less than $\mu$. At the end of this process, either the new largest index $j$ at which the maximum norm occurs is strictly less than $n$, or the maximum $\mu$ decreases. Of course, the latter cannot happen as $Q$ was assumed to be of minimal distance to $\mc{T}_n$.

Thus, we may assume that the largest index $j$ at which $\mu$ occurs is strictly less than $n$. In this case we have that $$f(a_j,a_{j+1})=\|Q_{j+1}\|^2<\|Q_j\|^2=f(a_{j-1},a_j).\smallskip$$ As above, we may invoke the Intermediate Value Theorem to obtain a root $a_j^\prime$ of the continuous function $$h(x)\coloneqq f(a_{j-1},x)-f(x,a_{j+1})$$ on the interval $\left[a_j,a_{j+1}\right]$. By replacing $a_j$ with $a_j^\prime$ in the sequence $\{a_i\}_{i=0}^n$, one may equate $\|Q_j\|$ and $\|Q_{j+1}\|$ while preserving all other norms $\|Q_i\|$. Since $a_j^\prime\geq a_j$, Lemma~\ref{f increasing in x, decreasing in y lemma} demonstrates that the new common value of $\|Q_j\|$ and $\|Q_{j+1}\|$ is strictly less than $\mu$. Thus, this process either decreases the largest index $j$ at which the maximum norm occurs, or reduces the value of $\mu$. Since this argument may be repeated for smaller and smaller values of $j$, eventually $\mu$ must decrease---a contradiction.
%%%One may repeat the above argument to successively reduce the Arveson seminorms of the projection $Q$. Eventually, either the maximum $\mu$ is decreased---which cannot happen, as the distance from $Q$ to the nilpotents was assumed to be minimal---or the maxiumum is 
%%%	
%%%	Observe that for any given $y\in[x_{k-1},x_{k+1}]$, one may replace the sequence $\{x_i\}_{i=0}^n$ with a new sequence $\{y_i\}_{i=0}^n$ defined by 
%%%	$$y_i=\left\{\begin{array}{ll}
%%%	x_i & \text{if}\,\,i\neq k\\
%%%	y & \text{if}\,\, i=k
%%%	\end{array}\right.$$
%%%	that describes a rank $n-1$ projection $Q^\prime$ satisfying $\|Q_i^\prime\|=\|Q_i^\prime\|$ for all $i\neq k,k+1$. Define By Lemma~\ref{f increasing in x, decreasing in y lemma}, we have that $f(y_{k-1},y_k)\leq 1$. 
	\end{proof}
	\smallskip
	
	\section{Computing the Distance}

	We will now utilize the results of \S2 to determine the precise value of $\nu_{n-1,n}$. The first step in this direction is the following proposition, which applies Theorem~\ref{Arveson seminorms must be equal theorem} to obtain a recursive description of the sequence $\{a_i\}_{i=0}^n$.

	\begin{prop}\label{recursive formula prop}
	
		Let $Q\in\mm_n(\cc)$ be a projection of rank $n-1$ that is of minimal distance to $\mc{T}_n$. If $\{a_i\}_{i=0}^n$ denotes the non-decreasing sequence from equation~(\ref{defining ak}), then
		$$a_k=\displaystyle{\frac{-\nu_{n-1,n}^4+2\nu_{n-1,n}^2a_{k-1}+\nu_{n-1,n}^2-a_{k-1}}{\nu_{n-1,n}^2a_{k-1}+\nu_{n-1,n}^2-a_{k-1}}}\smallskip$$
		for each $k\in\{1,2,\ldots, n\}$.
	
	\end{prop}

	\begin{proof}
		Since the distance from $Q$ to $\mc{T}_n$ is minimal, Theorems~\ref{Arveson distance formula} and \ref{Arveson seminorms must be equal theorem} imply that $\|E_{k-1}^\perp QE_k\|=\nu_{n-1,n}$ for all $k\in\{1,2,\ldots,n\}$. Thus, with $f$ as in Theorem~\ref{function definition of Arveson seminorm thm}, we have that
		$$f(a_{k-1},a_k)=\|E_{k-1}^\perp QE_k\|^2=\nu_{n-1,n}^2.\smallskip$$ The desired formula can now be obtained by solving this equation for $a_k$.
	\end{proof}
	\smallskip
	
%%%	\begin{rmk}
%%%	
%%%			Hmmm... I think it might make more sense to express the $q_{11}$ term in the function above as $1-a_1$. In this way the function is expressed entirely in terms of the sequence $\{a_i\}_{i=0}^n$. The new function would given by 
%%%			$$a_k=\displaystyle{\frac{a_1^2+2a_1a_{k-1}-a_1-a_{k-1}}{a_1a_{k-1}+a_1-1}}\vspace{0.1cm}$$
%%%			and everything that follows should go through (with minor modifications).
%%%			I will change this later.\smallskip
%%%	
%%%	\end{rmk}
		
%		\begin{cor}
%	
%		Let $Q=(q_{ij})$ be a rank $n-1$ projection of minimal distance to the nilpotents such that $q_{ij}\leq 0$ whenever $i\neq j$. Let $k\in\{1,2,\ldots, n\}$ be fixed, let $Q_k$ denote the $k^{th}$ Arveson corner of $Q$, and define $B=Q_k^*Q_k$. Then
%		\begin{equation}
%		\|Q_k\|^2=\frac{\mathrm{Tr}(B)+\sqrt{2\mathrm{Tr}(B^2)-\mathrm{Tr(B)}^2}}{2}
%		\end{equation}
%		and the equation 
%		\begin{equation}\label{formula for q_{kk}}
%		q_{kk}=\frac{q_{11}^2-x_{k-1}^2+q_{11}x_{k-1}^2}{q_{11}+q_{11}x_{k-1}-x_{k-1}}
%		\end{equation}
%		holds for all $k\in\{1,2\ldots, n\}$.	
%	
%	\end{cor}

	The recursive formula for $a_k$ described in Proposition~\ref{recursive formula prop} will be the key to computing $\nu_{n-1,n}$. Our goal will be to use this formula and some basic properties of the sequence $\{a_i\}_{i=0}^n$ to determine a list of candidates for $\nu_{n-1,n}^2$. A careful analysis of these candidates will reveal that exactly one of them satisfies a certain necessary norm inequality from \cite{MacDonaldIdempotents}. This value must therefore be $\nu_{n-1,n}^2$.\\

	To simplify notation, let $t=\nu_{n-1,n}^2$ and define the function $h_t:[0,1]\rightarrow\rr$ by 
	\begin{equation}\label{ht equation}
	h_t(x)=\frac{-t^2+2tx+t-x}{tx+t-x}.\smallskip
	\end{equation}
	Proposition~\ref{recursive formula prop} states that for each $k\in\{1,2,\ldots, n\}$,
	$$
	a_k=\frac{-t^2+2ta_{k-1}+t-a_{k-1}}{ta_{k-1}+t-a_{k-1}}=h_t(a_{k-1}).\smallskip
	$$
	Since $h_t(0)=(t-t^2)/t=1-t=a_1$, this formula may be expressed as
	$a_k=h_t^{(k)}(0)$
	for all $k\in\{1,2,\ldots, n\}$. Upon taking into account the condition $a_n=\mathrm{Tr}(P)=1$, we are interested in identifying the values of $t\in\left[\frac{1}{4},1\right]$ that satisfy the equation
	$h_t^{(n)}(0)=1.$
	
	Notice that each expression $h_t^{(k)}(0)$ is a rational function of $t$. For each $k\geq 1$, let $p_{k-1}(t)$ and $q_{k-1}(t)$ denote polynomials in $t$ such that 
	$$
	h_t^{(k)}(0)=\frac{p_{k-1}(t)}{q_{k-1}(t)}.\smallskip
	$$
	It then follows that 
	\begin{align*}
	\frac{p_k(t)}{q_k(t)}&=h_t\left(h_t^{(k)}(0)\right)=h_t\left(\frac{p_{k-1}(t)}{q_{k-1}(t)}\right)=\frac{-t^2q_{k-1}(t)+2tp_{k-1}(t)+tq_{k-1}(t)-p_{k-1}(t)}{tp_{k-1}(t)+tq_{k-1}(t)-p_{k-1}(t)},\smallskip
	\end{align*}
	and hence we obtain the relations
	\begin{equation}\label{p_k(t) recurrence}
	p_k(t)=t(1-t)q_{k-1}(t)+(2t-1)p_{k-1}(t),
	\end{equation}
	\begin{equation}\label{q_k(t) recurrence}
	q_k(t)=tq_{k-1}(t)-(1-t)p_{k-1}(t).\smallskip
	\end{equation}
	We may replace $p_{k-1}(t)$ in (\ref{q_k(t) recurrence}) using equation~(\ref{p_k(t) recurrence}), thereby leading to a recurrence expressed only in the $q_k(t)$'s. Specifically, we have that

	$$\begin{array}{rcl}
	q_k(t)&=&tq_{k-1}(t)-(1-t)p_{k-1}(t)\vspace{0.2cm}\\
	&=&tq_{k-1}(t)-(1-t)\left[t(1-t)q_{k-2}(t)+(2t-1)p_{k-2}(t)\right]\vspace{0.2cm}\\
	&=&tq_{k-1}(t)-t(1-t)^2q_{k-2}(t)-(2t-1)\left[tq_{k-2}(t)-q_{k-1}(t)\right]\vspace{0.2cm}\\
	&=&(3t-1)q_{k-1}(t)-t^3q_{k-2}(t)
	\end{array}\smallskip$$

\noindent for all integers $k\geq 2$. 
%We may extend this recurrence relation to include $k=1$ by choosing a suitable expression for $q_{-1}(t)$. Indeed, 
Notice as well that since $$\begin{array}{ccc}\displaystyle{\frac{p_0(t)}{q_0(t)}=h_t(0)=1-t} & \text{and} & \displaystyle{\frac{p_1(t)}{q_1(t)}=h_t(h_t(0))=\frac{-3t^2+4t-1}{-t^2+3t-1},}\end{array}\smallskip$$
we have initial terms $q_0(t)=1$ and $q_1(t)=-t^2+3t-1$. 
%Thus, we may write $q_1(t)=(3t-1)q_0(t)-t^3q_{-1}(t)$ by defining $q_{-1}(t)= t^{-1}$.
	
	The requirement that $h_t^{(n)}(0)=1$ is equivalent to asking that $p_{n-1}(t)=q_{n-1}(t)$. Using the relations above, this equation can be restated as $tq_{n-2}(t)=p_{n-2}(t)$, or equivalently $q_{n-1}(t)=t^2q_{n-2}(t)$ by (\ref{q_k(t) recurrence}). Thus, we wish to determine the values of $t\in\left[\frac{1}{4},1\right]$ that satisfy
	$$q_{n-1}(t)=t^2q_{n-2}(t),$$
	where $$\begin{array}{cccc}
	q_0(t)=1, & q_1=-t^2+3t-1, & \text{and} & q_k(t)=(3t-1)q_{k-1}(t)-t^3q_{k-2}(t)\,\,\,\text{for}\,\,\,k\geq 2.\end{array}$$
	
	 A solution to this problem will require closed-form expressions for the polynomials $q_{n-1}(t)$ and $q_{n-2}(t)$, which may be obtained via diagonalization arguments akin to those in \cite{MacDonaldProjections}. 
%%	  In order to obtain such expressions, we will first rewrite the recurrence relation defining these polynomials in terms of matrix multiplication:
%%$$
%%\begin{bmatrix}
%%q_k(t)\\
%%q_{k-1}(t)
%%\end{bmatrix}=\begin{bmatrix}
%%3t-1 & -t^3\\
%%1 & 0
%%\end{bmatrix}\begin{bmatrix}q_{k-1}(t)\\ q_{k-2}(t)\end{bmatrix}=\begin{bmatrix}
%%3t-1 & -t^3\\
%%1 & 0
%%\end{bmatrix}^k\begin{bmatrix}q_{0}(t)\\ q_{-1}(t)\end{bmatrix}.\smallskip
%%$$
%%One may therefore obtain a description of $q_{n-1}(t)$ and $q_{n-2}(t)$ by diagonalizing the matrix 
%%$$A\coloneqq\begin{bmatrix}
%%3t-1 & -t^3\\
%%1 & 0
%%\end{bmatrix}.\smallskip$$
%% \noindent Routine computations show that the eigenvalues of $A$ are given by 
%%$$\lambda_1=\frac{3t-1+(1-t)\sqrt{1-4t}}{2}=\frac{3t-1+(1-t)iy}{2}$$ and $$\lambda_2=\frac{3t-1-(1-t)\sqrt{1-4t}}{2}=\frac{3t-1-(1-t)iy}{2},\smallskip$$ 
%%
%%\noindent where $y\coloneqq\sqrt{4t-1}$. Furthermore, the columns of the matrix 
%%$P\coloneqq\begin{bmatrix}\lambda_1 & \lambda_2\\ 1 & 1\end{bmatrix}$ form a basis of eigenvectors corresponding to $\lambda_1$ and $\lambda_2$, respectively. By computing 
%%$$P^{-1}=\frac{1}{(1-t)iy}\begin{bmatrix}\phantom{-}1 & -\lambda_2\\
%%-1 & \phantom{-}\lambda_1\end{bmatrix}\smallskip$$
Our analysis reveals that with $$\begin{array}{cccc}
y\coloneqq\sqrt{4t-1}, & \displaystyle{\lambda_1\coloneqq\frac{3t-1+(1-t)iy}{2}},  & \text{and} &  \displaystyle{\lambda_2\coloneqq\frac{3t-1-(1-t)iy}{2},}\end{array}$$
we have
$$\begin{array}{ccc}
\displaystyle{q_{n-1}(t)=\frac{t\left(\lambda_1^{n}-\lambda_2^{n}\right)-\lambda_2\lambda_1^{n}+\lambda_1\lambda_2^{n}}{t(1-t)iy}} & \text{and} & \displaystyle{q_{n-2}(t)=\frac{t\left(\lambda_1^{n-1}-\lambda_2^{n-1}\right)-\lambda_2\lambda_1^{n-1}+\lambda_1\lambda_2^{n-1}}{t(1-t)iy}}.
\end{array}$$
% \noindent and setting $D=\mathrm{diag}(\lambda_1,\lambda_2)$, we have that $A=PDP^{-1}$. Consequently, 
%$$\begin{bmatrix} q_{n-1}(t)\vspace{0.1cm}\\ q_{n-2}(t)\end{bmatrix}=PD^{n-1}P^{-1}\begin{bmatrix}q_{0}(t)\vspace{0.1cm} \\ q_{-1}(t)\end{bmatrix}=\frac{1}{t(1-t)iy}\begin{bmatrix}
%t\left(\lambda_1^{n}-\lambda_2^{n}\right)-\lambda_2\lambda_1^{n}+\lambda_1\lambda_2^{n}\vspace{0.1cm} \\
%t\left(\lambda_1^{n-1}-\lambda_2^{n-1}\right)-\lambda_2\lambda_1^{n-1}+\lambda_1\lambda_2^{n-1}
%\end{bmatrix}.\smallskip$$

 These expressions for $q_{n-1}(t)$ and $q_{n-2}(t)$ can now be used to identify the desired values of $t$. Indeed, when $q_{n-1}(t)=t^2q_{n-2}(t)$, we have that \vspace{0.1cm}
$$\begin{array}{crcl}
&t\left(\lambda_1^{n}-\lambda_2^{n}\right)-\lambda_2\lambda_1^{n}+\lambda_1\lambda_2^{n}&=&t^2\left(t\left(\lambda_1^{n-1}-\lambda_2^{n-1}\right)-\lambda_2\lambda_1^{n-1}+\lambda_1\lambda_2^{n-1}\right)\vspace{0.2cm} \\ 
\Longrightarrow & \lambda_1^n(t-\lambda_2)-\lambda_2^n(t-\lambda_1)&=&t^2\left(\lambda_1^{n-1}(t-\lambda_2)-\lambda_2^{n-1}(t-\lambda_1)\right)\vspace{0.2cm} \\ 
\Longrightarrow & \lambda_2^{n-1}(t^2-\lambda_2)(t-\lambda_1)&=&\lambda_1^{n-1}(t^2-\lambda_1)(t-\lambda_2),\vspace{0.2cm}
\end{array}$$
and therefore\smallskip
\begin{equation}\label{t and lambda product is 1}
\displaystyle{\left(\frac{\lambda_2}{\lambda_1}\right)^{n-1}\left(\frac{t^2-\lambda_2}{t^2-\lambda_1}\right)\left(\frac{t-\lambda_1}{t-\lambda_2}\right)}=1.\smallskip
\end{equation}
%
%\begin{equation}\label{Ending equation 1}
%(t-\lambda_2)\lambda_1^n-(t-\lambda_1)\lambda_2^n=t^2\left((t-\lambda_2)\lambda_1^{n-1}-(t-\lambda_1)\lambda_2^{n-1}\right).
%\end{equation}

\noindent This equation may be simplified using the following identities that relate the values of $t$, $\lambda_1$, and  $\lambda_2$. Verification of these identities is straightforward, and thus their proofs are left to the reader.\smallskip

%%%\begin{lem}
%%%
%%%	If $y=\sqrt{4t-1}$, $\lambda_1=(3t-1+(1-t)iy)/2$, and $\lambda_2=(3t-1-(1-t)iy)/2$, then 
%%%	
%%%	\begin{itemize}\label{identities lemma}
%%%	
%%%	\item[\upshape{(i)}] \begin{equation}\label{Reduction identity 1}
%%%t-\lambda_1=(1-t)\left(\frac{1-iy}{2}\right)\,\,\,\,\,\text{and}\,\,\,\,\,t-\lambda_2=(1-t)\left(\frac{1+iy}{2}\right);
%%%\end{equation}
%%%	
%%%	\item[\upshape{(ii)}]\begin{equation}
%%%	t^2-\lambda_1=(1-t)\left(\frac{1-2t-iy}{2}\right)\,\,\,\,\,\text{and}\,\,\,\,\,t^2-\lambda_2=(1-t)\left(\frac{1-2t+iy}{2}\right);
%%%	\end{equation}
%%%	
%%%	\item[\upshape{(iii)}]\begin{equation}
%%%	\frac{1+iy}{1-iy}=\frac{1-2t+iy}{2t}\,\,\,\,\,\text{and}\,\,\,\,\,\frac{1-iy}{1+iy}=\frac{1-2t-iy}{2t};
%%%	\end{equation}
%%%	
%%%	\item[\upshape{(iv)}]\begin{equation}
%%%	\frac{\lambda_2}{\lambda_1}=\left(\frac{1+iy}{1-iy}\right)^3\smallskip
%%%	\end{equation}
%%%	\end{itemize}
%%%
%%%\end{lem}

\begin{lem}\label{four identities lemma}

	If $y=\sqrt{4t-1}$, $\lambda_1=(3t-1+(1-t)iy)/2$, and $\lambda_2=(3t-1-(1-t)iy)/2$, then \vspace{0.1cm}
	
	\begin{itemize}	
	\item[(i)] $\begin{array}{rcl}
	\displaystyle{t-\lambda_1=(1-t)\left(\frac{1-iy}{2}\right)} & and & \displaystyle{t-\lambda_2=(1-t)\left(\frac{1+iy}{2}\right)};
	\end{array}$\\ \vspace{0.1cm}
	
	\item[(ii)] $\begin{array}{ccc}
	\displaystyle{t^2-\lambda_1=(1-t)\left(\frac{1-2t-iy}{2}\right)} & and & \displaystyle{t^2-\lambda_2=(1-t)\left(\frac{1-2t+iy}{2}\right)};
	\end{array}$\\ \vspace{0.1cm}
	
	\item[(iii)] $\begin{array}{ccc}
	\displaystyle{\frac{1+iy}{1-iy}=\frac{1-2t+iy}{2t}} & and & \displaystyle{\frac{1-iy}{1+iy}=\frac{1-2t-iy}{2t}};
	\end{array}$\\ \vspace{0.1cm}
	
	\item[(iv)] $\begin{array}{c}\displaystyle{\frac{\lambda_2}{\lambda_1}=\left(\frac{1+iy}{1-iy}\right)^3}\end{array}$.\smallskip
	
	\end{itemize}	
	
\end{lem}

%\begin{proof}
%
%	Verification of statements (i)-(iii) is straightforward, and thus their proofs are left to the reader. For (iv), an application of the Binomial Theorem demonstrates that
%	$$\begin{array}{rclcl}
%	(1+iy)^3 & = & 1+3iy-3y^2-iy^3\vspace{0.2cm}\\
%	& = & (1-3y^2)+iy(3-y^2)\vspace{0.2cm}\\
%	& = & 4(1-3t)+4(1-t)iy &=&-8\lambda_2.
%	\end{array}$$
%From this we deduce that $(1-iy)^3=\overline{(1+iy)^3}=-8\overline{\lambda_2}=-8\lambda_1,$
%and thus the result holds.
%\end{proof}
%\smallskip

One may apply the identities above to simplify equation~(\ref{t and lambda product is 1}) as follows:
$$\begin{array}{rclcl}
1&=&\displaystyle{\left(\frac{\lambda_2}{\lambda_1}\right)^{n-1}\left(\frac{t^2-\lambda_2}{t^2-\lambda_1}\right)\left(\frac{t-\lambda_1}{t-\lambda_2}\right)}\vspace{0.2cm} \\ 
&=&\displaystyle{\left(\frac{1+iy}{1-iy}\right)^{3(n-1)}\left(\frac{1-2t+iy}{1-2t-iy}\right)\left(\frac{1-iy}{1+iy}\right)}\vspace{0.2cm} \\
&=&\displaystyle{\left(\frac{1+iy}{1-iy}\right)^{3n-3}\left(\frac{1+iy}{1-iy}\right)^2\left(\frac{1-iy}{1+iy}\right)}\vspace{0.2cm} \\
&=&\displaystyle{\left(\frac{1+iy}{1-iy}\right)^{3n-2}}.\vspace{0.1cm}
\end{array}$$
%%%
%%%$$\begin{array}{rclcl}
%%%\begin{displaystyle}\frac{1+iy}{1-iy}\end{displaystyle}&=&\begin{displaystyle}\left(\frac{\lambda_2}{\lambda_1}\right)^{n-1}\left(\frac{\lambda_2-t^2}{\lambda_1-t^2}\right)\end{displaystyle}\\ 
%%%&=&\begin{displaystyle}\left(\frac{\lambda_2}{\lambda_1}\right)^{n-1}\left(\frac{1-2t+iy}{1-2t-iy}\right)\end{displaystyle}\\
%%%&=&\begin{displaystyle}\left(\frac{\lambda_2}{\lambda_1}\right)^{n-1}\left(\frac{1+iy}{1-iy}\right)^2\end{displaystyle}&=&\begin{displaystyle}\left(\frac{1+iy}{1-iy}\right)^{3n-1}\end{displaystyle}.
%%%\end{array}$$
We therefore conclude that   
$\displaystyle{\frac{1+iy}{1-iy}}=\rho_m^k,$
where $m\coloneqq 3n-2$, $\rho_m\coloneqq\displaystyle{e^{2\pi i/m}}$, and $k$ is an integer.

We are now in a position to determine the possible values of $t$. By solving for $y$ in the equation above, we obtain 
$$\begin{array}{rcccl}
y&=&\displaystyle{\frac{1}{i}\frac{\rho_m^k-1}{\rho_m^k+1}}&=&\displaystyle{\frac{1}{i}\frac{\rho_m^{k/2}\left(\rho_m^{k/2}-\rho_m^{-k/2}\right)}{\rho_m^{k/2}\left(\rho_m^{k/2}+\rho_m^{-k/2}\right)}}\vspace{0.2cm}\\
&&&=&\displaystyle{\frac{\rho_m^{k/2}-\rho_m^{-k/2}}{2i}\frac{2}{\rho_m^{k/2}+\rho_m^{-k/2}}}\vspace{0.2cm}\\
&&&=&\displaystyle{\frac{\sin\left(k\pi/m\right)}{\cos\left(k\pi/m\right)}\,=\,\tan\left(\frac{k\pi}{m}\right)}.\end{array}\smallskip$$
 Since $y=\sqrt{4t-1}$, we have 
$$
t=\frac{1}{4}\left(\tan^2\left(\frac{k\pi}{m}\right)+1\right)=\frac{1}{4}\sec^2\left(\frac{k\pi}{3n-2}\right)\,\,\,\text{for some}\,\,k\in\zz.
\smallskip$$
 \noindent That is, the distance $\nu_{n-1,n}$ from $Q$ to $\mc{T}_n$ must belong to the set $\left\{\frac{1}{2}\sec\left(\frac{k\pi}{3n-2}\right):k\in\zz\right\}.\smallskip$

It remains to determine which element of this set represents $\nu_{n-1,n}$. We will accomplish this task by appealing to the following result of MacDonald concerning a lower bound on the distance from a projection to a nilpotent.\smallskip

\begin{prop}\textup{\cite[Lemma 3.3]{MacDonaldIdempotents}}\label{distance estimate prop}
If $P\in\mm_n(\cc)$ is a projection of rank $r$ and $N\in\mm_n(\cc)$ is nilpotent, then 
$$\|P-N\|\geq\sqrt{\frac{r}{2n}\left(1+\frac{r}{n}\right)}.$$
\end{prop}

In the analysis that follows, we will demonstrate that the only value in $\left\{\frac{1}{2}\sec\left(\frac{k\pi}{3n-2}\right):k\in\zz\right\}$ that respects the lower bound of Proposition~\ref{distance estimate prop} for projections of rank $r=n-1$ occurs when $k=n-1$. We begin with the following lemma, which proves that MacDonald's lower bound is indeed satisfied for this choice of $k$.\smallskip

\begin{lem}\label{Inequality works when k=1}
For every integer $n\geq 3$, 
$$\frac{n-1}{2n}\left(1+\frac{n-1}{n}\right)\leq \frac{1}{4}\sec^2\left(\frac{(n-1)\pi}{3n-2}\right)\leq 1$$
\end{lem}

\begin{proof}

	Define $\alpha_n\coloneqq (3n-2)/(n-1)$. By considering reciprocals, this problem is equivalent to that of establishing the inequalities
	$$
	\begin{array}{rl}\displaystyle{\frac{1}{4}\leq\cos^2\left(\frac{\pi}{\alpha_n}\right)\leq\frac{n^2}{2(n-1)(2n-1)}} & \text{for all}\,\,n\in\zz, n\geq 3.\end{array}$$
	In the computations that follow, it will be helpful to view $n$ as a continuous variable on $[3,\infty)$.
	
	To establish the inequality $1/4\leq\cos^2\left(\pi/\alpha_n\right),$ simply note that $\pi/\alpha_n$ is an increasing function of $n$ tending to $\pi/3$, $\cos(x)$ is decreasing on $[0,\pi/3]$, and $\cos(\pi/3)=1/2$. The second inequality will require a bit more work. Since $(2n-\frac{3}{2})^2\geq 2(n-1)(2n-1)$ for all $n$, it suffices to prove that
	$$
	\begin{array}{rl}\displaystyle{\cos^2\left(\frac{\pi}{\alpha_n}\right)\leq\frac{n^2}{\left(2n-\frac{3}{2}\right)^2}} &\text{for}\,\,n\in[3,\infty).
	\end{array}$$ 
	
	\noindent Note that this inequality holds if and only if the function
	$$\begin{array}{c}
	\displaystyle{f(n)\coloneqq\frac{2n}{4n-3}-\cos\left(\frac{\pi}{\alpha_n}\right)}
	\end{array}$$ 
	is non-negative on $[3,\infty)$.
	
	We will prove that $f^\prime(n)<0$ for all $n\in[3,\infty)$, so that $f$ is decreasing on this interval. Since $$\begin{array}{ccc}\displaystyle{\lim_{n\rightarrow\infty}f(n)=0} & \text{and} & \displaystyle{f(3)=\frac{2}{3}-\cos\left(\frac{2\pi}{7}\right)\approx 0.043>0,}\end{array}\smallskip$$
	 \noindent this will demonstrate that $f(n)\geq 0$ for all $n\geq 3$. To this end, we compute
	$$f^\prime(n)=\displaystyle{\frac{16\pi\sin\left(\frac{\pi}{\alpha_n}\right)n^2-24\pi\sin\left(\frac{\pi}{\alpha_n}\right)n+9\pi\sin\left(\frac{\pi}{\alpha_n}\right)-54n^2+72n-24}{(4n-3)^2(3n-2)^2}.}\smallskip$$
		Of course $(4n-3)^2(3n-2)^2\geq 0$, so the sign of $f^\prime(n)$ depends only on the sign of
	$$g(n)\coloneqq \displaystyle{16\pi\sin\left(\frac{\pi}{\alpha_n}\right)n^2-24\pi\sin\left(\frac{\pi}{\alpha_n}\right)n+9\pi\sin\left(\frac{\pi}{\alpha_n}\right)-54n^2+72n-24}.\smallskip$$
	 But since $\pi/\alpha_n\in[\pi/4,\pi/3]$ for $n\geq 3$, we have that $\sin\left(\pi/\alpha_n\right)\in[\sqrt{2}/2,\sqrt{3}/{2}]$ for all such $n$, and hence
	\begin{align*}
	g(n)&\leq 16\pi\left(\frac{\sqrt{3}}{2}\right)n^2-24\pi\left(\frac{\sqrt{2}}{2}\right)n+9\pi\left(\frac{\sqrt{3}}{2}\right)-54n^2+72n-24\\
	&=\left(8\sqrt{3}\pi-54\right)n^2-\left(12\sqrt{2}-72\right)n+\left(\frac{9\sqrt{3}}{2}-24\right).
	\end{align*}
	This upper bound for $g$ is a concave quadratic whose larger root occurs at $n\approx 1.8105.$ It follows that $g$ is negative on $[3,\infty)$, and therefore so too is $f^\prime$.
\end{proof}
\smallskip

\begin{lem}\label{Calculus lem}

	For any integer $n\geq 3$, the set 
	$$\left\{\frac{1}{4}\sec^2\left(\frac{k\pi}{3n-2}\right):k\in\zz\right\}\smallskip$$
contains exactly one value in $\displaystyle{\left[\frac{n-1}{2n}\left(1+\frac{n-1}{n}\right),1\right]}$, and it occurs when $k=n-1$.\\
\end{lem}

\begin{proof}
Fix an integer $n\geq 3$. We wish to prove that 
$\mc{A}\coloneqq \left\{\cos^2\left(\frac{k\pi}{3n-2}\right):k\in\zz\right\}$ contains exactly one value in the interval
$\left[\frac{1}{4},\frac{n^2}{2(n-1)(2n-1)}\right].$
Since Lemma~\ref{Inequality works when k=1} demonstrates that this is the case when $k=n-1$, it suffices to show that no other values in $\mathcal{A}$ are within distance
\begin{align*}
\beta(n)\coloneqq\frac{n^2}{2(n-1)(2n-1)}-\frac{1}{4}
\end{align*} 
of $\cos^2((n-1)\pi/(3n-2))$.
%
%Note that $\delta$ is a decreasing function of $n$ on $[3,\infty)$, and hence $\delta(n)\leq\alpha(3)=1/5$.

Note, however, that not all values of $k\in\mathbb{Z}$ need to be considered. In particular, since the function
%Indeed, one can see immediately that $k=0$ will yield a value of $1$, which is certainly outside the target range:
%\begin{align*}
%\left|\cos^2(\pi/\alpha)-1\right|&\geq 1-\cos^2(\pi/3)=\frac{1}{2}\geq\frac{1}{2(n-1)}-\frac{1}{4(2n-1)}
%\end{align*}
%\begin{align*}
%\cos^2\left(\frac{(k+(3n-2))\pi}{3n-2}\right)&=\cos^2\left(\frac{k\pi}{3n-2}\right),
%\end{align*}
 $k\mapsto\cos^2(k\pi/(3n-2))$ is periodic, it suffices to check only its values at the integers $k\in\{0,1,\ldots, 3n-2\}$. Additionally, since
$$\cos^2\left(\frac{((3n-2)-k)\pi}{3n-2}\right)=\cos^2\left(\frac{k\pi}{3n-2}\right)\,\,\,\text{for all}\,\,k,\smallskip$$ we may restrict our attention to $k\in\{0,1,2,\ldots, \floor*{(3n-2/2)}\}$. 

Although we are solely concerned with the integer values of $k$ described above, it will be useful to view $k$ as a continuous real variable. With this in mind, define the function $f_n:[0,(3n-2)/2]\rightarrow\rr$ by   
$$f_n(k)\coloneqq\sin\left(\frac{(n-k-1)\pi}{3n-2}\right)\sin\left(\frac{(n+k-1)\pi}{3n-2}\right).\smallskip$$
It follows from the identity $\cos^2(x)-\cos^2(y)=-\sin(x-y)\sin(x+y)$ that 
$$\begin{array}{ccc}\displaystyle{\left|\cos^2\left(\frac{k\pi}{3n-2}\right)-\cos^2\left(\frac{(n-1)\pi}{3n-2}\right)\right|<\beta(n)}&\Longleftrightarrow&\displaystyle{|f_n(k)|<\beta(n).}\end{array}\smallskip$$

%%%f_n^\prime(k)=\frac{\pi}{3n-2}\left(\sin\left(\frac{(n-k-1)\pi}{3n-2}\right)\cos\left(\frac{(n+k-1)\pi}{3n-2}\right)-\sin\left(\frac{(n+k-1)\pi}{3n-2}\right)\cos\left(\frac{(n-k-1)\pi}{3n-2}\right)\right)\\
\noindent Notice, however, that 
$$f_n^\prime(k)=\left(\frac{-\pi}{3n-2}\right)\sin\left(\frac{2k\pi}{3n-2}\right),\smallskip$$
so $f_n^\prime(k)<0$ on $[0,(3n-2)/2]$, and hence $f_n$ is decreasing on its domain. Since $f_n(n-1)=0$, it therefore suffices to prove that $f_n(n-2)> \beta(n)$ and $-f_n(n)> \beta(n).$ We will demonstrate that these inequalities hold via application of Taylor's Theorem. 

Consider the approximation of $\sin(x)$ by $x-x^3/6$, its third degree MacLauren polynomial. On $[0,\pi/6]$, the error in this approximation is at most
$$E(x)=\frac{\sin(\pi/6)}{4!}|x|^4=\frac{x^4}{48}.\smallskip$$ 
Thus, since $1/n\leq \pi/(3n-2)\leq \pi/6$, we have
$$\sin\left(\frac{\pi}{3n-2}\right)\geq\sin\left(\frac{1}{n}\right)\geq \left(\frac{1}{n}-\frac{1}{6n^3}-E\left(\frac{1}{n}\right)\right)\geq\left(\frac{1}{n}-\frac{1}{6n}-\frac{1}{48n}\right)=\frac{13}{16n}.\smallskip$$
It is routine to verify that $\sin\left((2x-1)\pi/(3x-2)\right)$ is an increasing function of $x$ on $[3,\infty)$. Consequently, this function is bounded below by $\sin\left(5\pi/7\right)$, its value at $x=3$. We deduce that 
$$-f_n(n)=\sin\left(\frac{\pi}{3n-2}\right)\sin\left(\frac{(2n-1)\pi}{3n-2}\right)\geq \frac{13}{16n} \sin\left(\frac{5\pi}{7}\right)\geq\frac{13}{16n}\cdot\frac{3}{4}=\frac{39}{64n}.\smallskip$$
Lastly, one may show directly that $$\frac{39}{64n}>\beta(n)\,\,\,\text{whenever}\,\,\,n> \frac{101+\sqrt{5521}}{60}\approx 2.9217,\smallskip$$ and hence $-f_n(n)> \beta(n)$ for our fixed integer $n\geq 3$. 

A similar analysis may now be used to prove that $f_n(n-2)> \beta(n)$. Indeed, it is straightforward to verify that $\sin\left((2n-3)\pi/(3n-2)\right)$ is bounded below by $\sin\left(2\pi/3\right)$, and therefore
$$
\begin{array}{rcl}
f_n(n-2)&=&\displaystyle{\sin\left(\frac{\pi}{3n-2}\right)\sin\left(\frac{(2n-3)\pi}{3n-2}\right)}\vspace{0.2cm}\\
&\geq&\displaystyle{\frac{13}{16n} \sin\left(\frac{2\pi}{3}\right)=\frac{13}{16n}\cdot\frac{\sqrt{3}}{2}\geq\frac{13}{16n}\cdot\frac{3}{4}=\frac{39}{64n}.}
\end{array}\smallskip$$
It now follows from the arguments of the previous case that $f_n(n-2)> \beta(n)$.
\end{proof}
\smallskip

With the above analysis complete, we may now present the main result of this paper: the distance from a projection in $\mm_n(\cc)$ of rank $n-1$ to the set $\mc{N}(\cc^n)$ is $$\nu_{n-1,n}=\frac{1}{2}\sec\left(\frac{(n-1)\pi}{3n-2}\right).\smallskip$$ Interestingly, this expression can be rewritten to bear an even stronger resemblance to MacDonald's formula in the rank-one case. \smallskip

\begin{thm}\label{main theorem}
	For every integer $n\geq 2$, the distance from the set of projections in $\mm_n(\cc)$ of rank $n-1$ to $\mc{N}(\cc^n)$ is
	$$\nu_{n-1,n}=\frac{1}{2}\sec\left(\frac{\pi}{\frac{n}{n-1}+2}\right).$$
\end{thm}

	\section{Closest Projection-Nilpotent Pairs}\label{Section: Closest proj-nil pairs}

Given a projection $Q$ in $\mm_n(\cc)$ of rank $n-1$ that is of distance $\nu_{n-1,n}$ to $\mc{T}_n$, the following theorem provides a means for determining an element $T\in\mc{T}_n$ that is closest to $Q$. As we will see in Theorem~\ref{summarizing thm}, this element of $\mc{T}_n$ is unique to $Q$.\smallskip

\begin{thm}\textup{\cite{Bini,MacDonaldIdempotents}}\label{Q-N is a multiple of a unitary}
	Fix $\gamma\in[0,\infty)$. An operator $A\in\mm_n(\cc)$ is such that $\|E_{i-1}^\perp AE_i\|=\gamma $ for all\linebreak $i\in\{1,2,\ldots, n\}$ if and only if there exist $T\in\mc{T}_n$ and a unitary $U\in\mm_n(\cc)$ such that $A-T=\gamma U$. Furthermore, if ${\|E_{i-1}^\perp AE_i\|=\gamma}$ and $\|E_i^\perp AE_i\|<\gamma$ for all $i\in\{1,2,\ldots, n-1\}$, then the operators $T$ and $U$ are unique.\smallskip

\end{thm}

With this result in hand, we are now able to describe all closest pairs $(Q,N)$ where $Q$ is a projection of rank $n-1$ and $N\in\mc{N}(\cc^n).$

\begin{thm}\label{summarizing thm}
Fix a positive integer $n\geq 2$. Let $\{a_i\}_{i=0}^n$ be the sequence given by $a_0=0$ and 
$$\begin{array}{rl}a_k=\displaystyle{\frac{-\nu_{n-1,n}^4+2\nu_{n-1,n}^2a_{k-1}+\nu_{n-1,n}^2-a_{k-1}}{\nu_{n-1,n}^2a_{k-1}+\nu_{n-1,n}^2-a_{k-1}}} & \text{for}\,\,\,k\geq 1.\end{array}\smallskip$$
  
\noindent Let $\{z_i\}_{i=1}^n$ be a sequence of complex numbers of modulus $1$, define $$e=\begin{bmatrix}
z_1\sqrt{a_1-a_0} & z_2\sqrt{a_2-a_1} & \cdots & z_n\sqrt{a_n-a_{n-1}}
\end{bmatrix}^T,\smallskip$$ and let $Q=I-e\otimes e^*$.\smallskip

\begin{itemize}

\item[(i)]$Q$ is a projection of rank $n-1$ such that $\mathrm{dist}(Q,\mc{T}_n)=\nu_{n-1,n}$. Moreover, every projection of rank $n-1$ that is of minimal distance to $\mc{T}_n$ is of this form.\smallskip

\item[(ii)]There is a unique operator $T\in\mc{T}_n$ of minimal distance to $Q$, and this $T$ is such that $Q-T=\nu_{n-1,n} U$ for some unitary $U\in\mm_n(\cc)$. Thus, if $q_k=Qe_k$ and $t_k=Te_k$ denote the columns of $Q$ and $T$, respectively, then one can iteratively determine columns $t_k$ by solving the system of linear equations $$\left\{\begin{array}{rcl}
\langle q_1-t_1,q_k-t_k\rangle & = & 0\\
\langle q_2-t_2,q_k-t_k\rangle & = & 0\\
\vdots\hspace{1cm} &  & \vdots\\
\langle q_{k-1}-t_{k-1},q_k-t_k\rangle & = & 0\end{array}\right.$$
for $k\in\{2,3,\ldots, n\}$.

\end{itemize}
\end{thm}

\begin{proof}

	Statement (i) follows immediately from the results of $\S2$ and $\S3$. For statement (ii), the existence of $T$ and $U$ is guaranteed by Theorems~\ref{Arveson seminorms must be equal theorem} and \ref{Q-N is a multiple of a unitary}. All that remains to show is the uniqueness of these operators.
	
 	To accomplish this task, note that it suffices to prove uniqueness in the case that $z_i=1$ for all $i$ (i.e., when $q_{ij}\leq 0$ for all $i\neq j$). For $k\in\{1,2,\ldots, n\}$, let $Q_k$ denote the restriction of $E_{k-1}^\perp QE_k$ to the range of $E_k$, and define $B_k\coloneqq Q_k^*Q_k$. Let $Q_k^\prime=E_k^\perp Q_k$, so that
	$$Q_k=\begin{bmatrix}
	\begin{array}{cc}
	v_k^* \vspace{-0.25cm}\\
	\rule{0.7cm}{0.01cm}\end{array}\\
	Q_k^\prime
	\end{bmatrix},$$
where $v_k\coloneqq\begin{bmatrix}
	q_{k1} & q_{k2} & \ldots & q_{kk}
	\end{bmatrix}^T$.
		
We will demonstrate that $\|Q_k^\prime\|<\|Q_k\|$ for all $k\in\{1,2,\ldots,n-1\}$, and therefore obtain the uniqueness of $T$ and $U$ via Theorem~\ref{Q-N is a multiple of a unitary}. Observe that this inequality holds when $k=1$, as $$\|Q_1\|^2-\|Q_1^\prime\|^2=q_{11}^2=\nu_{n-1,n}>0.\smallskip$$ Suppose now that $k\in\{2,3,\ldots, n-1\}$ is fixed, and define $B_k^\prime\coloneqq {Q_k^\prime}^*Q_k^\prime=B_k-v_kv_k^*.$ One may determine the entries of $B_k^\prime=(b_{ij}^\prime)$ using the formulas for the entries of $B_k=(b_{ij})$ from Lemma~\ref{Finding the entries and eigenvalues of Q_k^*Q_k}~(i). Indeed, 
	$$\begin{array}{rclcl}
	b_{kk}^\prime&=&\displaystyle{b_{kk}-q_{kk}^2}\vspace{0.1cm}\\
	&=&\displaystyle{q_{kk}-a_{k-1}(1-q_{kk})-q_{kk}^2}\vspace{0.1cm}\\
	&=&\displaystyle{(q_{kk}-a_{k-1})(1-q_{kk})}\hspace{0.2cm} = \hspace{0.2cm} \displaystyle{(1-a_k)(1-q_{kk})},
	\end{array}$$
	
	\noindent and for if $i<k$,
	$$\begin{array}{rclcl}
	b_{ii}^\prime&=&\displaystyle{b_{ii}-q_{ki}^2}\vspace{0.1cm}\\
	&=&\displaystyle{(1-a_{k-1})(1-q_{ii})-(1-q_{kk})(1-q_{ii})}\vspace{0.1cm}\\
	&=&\displaystyle{(q_{kk}-a_{k-1})(1-q_{ii})}\hspace{0.2cm} = \hspace{0.2cm} \displaystyle{(1-a_k)(1-q_{ii})}.
	\end{array}$$
	
	\noindent If $i,j,$ and $k$ are all distinct, then	
	$$\begin{array}{rclcl}
	b_{ij}^\prime&=&\displaystyle{b_{ij}-q_{ki}q_{kj}}\vspace{0.1cm}\\
	&=&\displaystyle{-(1-a_{k-1})q_{ij}+q_{ij}(1-q_{kk})}\vspace{0.1cm}\\
	&=&\displaystyle{-(q_{kk}-a_{k-1})q_{ij}}\vspace{0.1cm}\hspace{0.2cm}=\hspace{0.2cm}\displaystyle{-(1-a_k)q_{ij}.}
	\end{array}$$
	
	\noindent Finally, either $i<j=k$ or $j<i=k$. In the case of former, we have
	$$\begin{array}{rclcl}
	b_{ik}^\prime&=&\displaystyle{b_{ik}-q_{ki}q_{kk}}\vspace{0.1cm}\\
	&=&\displaystyle{a_{k-1}q_{ik}-q_{ik}q_{kk}}\vspace{0.1cm}\\
	&=&\displaystyle{-(q_{kk}-a_{k-1})q_{ik}}\hspace{0.2cm} = \hspace{0.2cm} \displaystyle{-(1-a_k)q_{ik}}.
	\end{array}\smallskip$$
	\noindent The fact that $B_k^\prime$ is self-adjoint implies that $b^\prime_{kj}=-(1-a_k)q_{kj}$ for all $j<k$ as well. 
	
	The above expressions for the entries $b_{ij}^\prime$ reveal that 
	$B_k^\prime=(1-a_k)(I-\widehat{Q}),$
	where $\widehat{Q}\in\mm_k(\cc)$ denotes the $k^{th}$ leading principal submatrix of $Q$. 
%	Since $Q$ has rank $n-1$, 
%	Corollary~\ref{Eigenvalue multiplicity corollary}~(i) implies that $\lambda=1$ occurs as an eigenvalue of $\widehat{Q}$ with multiplicity at least $k-1$, and hence $0$ occurs as an eigenvalue of $B_k^\prime$ with multiplicity at least $k-1$. 
%	
%	
	Since $Q$ has rank $n-1$, it follows that $I-\widehat{Q}$ has rank at most $1$, and hence $B_k^\prime$ has at most one non-zero eigenvalue. Consequently,
	$$\|B_k^\prime\|=\mathrm{Tr}(B_k^\prime)=\sum_{\ell=1}^{k}(1-a_{k})(1-q_{\ell\ell})=a_k(1-a_{k}).$$
	
Now let $f:[0,1]\times[0,1]\rightarrow\rr$ denote the function from Theorem~\ref{function definition of Arveson seminorm thm}, so that $\|Q_k\|^2=f(a_{k-1},a_k)$. Suppose for the sake of contradiction that $\|B_k\|=\|B_k^\prime\|$, and hence $f(a_{k-1},a_k)=a_k(1-a_k)$. One may verify that for this equation to hold, we necessarily have that $a_{k}=1$ or $a_k=a_{k-1}$. 
 
 If the former is true, then $a_j=1$ for all $j\geq k$. In particular, $a_{n-1}=a_n$. From this it follows that $q_{nn}=1-(a_n-a_{n-1})=1$, and hence $\|Q_n\|\geq 1$. This contradicts the minimality of $\mathrm{dist}(Q,\mc{T}_n)$. If instead $a_k=a_{k-1}$, then $q_{kk}=1$, and thus $\|Q_k\|\geq 1$. Again we reach a contradiction. We therefore conclude that $\|B_k^\prime\|<\|B_k\|$, hence $\|Q_k^\prime\|<\|Q_k\|$.
\end{proof}
\smallskip

To save the reader from lengthy computations, we have included a few examples of pairs $(Q,T)$ where $Q\in\mm_n(\cc)$ is a projection of rank $n-1$, $T$ belongs to $\mc{T}_n$, and $\|Q-T\|=\nu_{n-1,n}$. Theorem~\ref{summarizing thm} implies that if $(Q^\prime,T^\prime)$ is any other projection-nilpotent pair such that $rank(Q^\prime)=n-1$ and $\|Q^\prime-T^\prime\|=\nu_{n-1,n}$, then there is a unitary $V\in\mm_n(\cc)$ such that  $Q^\prime=V^*QV$ and $T^\prime=V^*TV$. In each case the entries of $Q$ and $T$ have been rounded to the fifth decimal place.\\

\noindent\underline{$n=3$}\\

\vspace{-0.3cm}
$\begin{array}{l}
Q=\begin{bmatrix}
\phantom{-}0.64310 & -0.31960 & -0.35689\\
-0.31960 & \phantom{-}0.71379 & -0.31960\\
-0.35689 & -0.31960 & \phantom{-}0.64310
\end{bmatrix},\vspace{0.3cm} \\  
T=\begin{bmatrix}
0 & -0.49697 & -0.80194\\
0 & 0 & -0.49697\\
0 & 0 & 0
\end{bmatrix};\smallskip
\end{array}\vspace{0.75cm}$ 

\noindent\underline{$n=4$}\\

\vspace{-0.3cm}
$
\begin{array}{l}
Q=\begin{bmatrix}
  \phantom{-}0.72361  & -0.24860 & -0.24860& -0.27639\\
  -0.24860& \phantom{-}0.77639 & -0.22361& -0.24860\\
  -0.24860& -0.22361& \phantom{-}0.77639 &-0.24860 \\
 -0.27639 & -0.24860& -0.24860& \phantom{-}0.72361
\end{bmatrix},\vspace{0.3cm}\\
T=\begin{bmatrix}
0 & -0.34356 & -0.46094 & -0.65836\\
0 & 0 & -0.34164 & -0.46094 \\
0 & 0 & 0 & -0.34356 \\
0 & 0 & 0 &  0
\end{bmatrix};
\end{array}\vspace{0.75cm} $

\noindent\underline{$n=5$}\\

\vspace{-0.3cm}
$
\begin{array}{l}
Q=\begin{bmatrix}
\phantom{-}0.77471 & -0.20512 & -0.19907 & -0.20512 & -0.22528\\
-0.20512 & \phantom{-}0.81324 & -0.18126 & -0.18676 & -0.20512\\
-0.19907 & -0.18126 & \phantom{-}0.82409 & -0.18126 & -0.19907 \\
 -0.20512 & -0.18676 & -0.18126 & \phantom{-}0.81324 & -0.20512 \\
-0.22528 & -0.20512 & -0.19907 & -0.20512 & \phantom{-}0.77472
\end{bmatrix},\vspace{0.3cm} \\
T=\begin{bmatrix}
0 & -0.26477 & -0.32678 & -0.41846 & -0.55566\\
0 & 0 & -0.26373 & -0.32453 & -0.41846 \\
0 & 0 & 0 & -0.26373 & -0.32678 \\
0 & 0 & 0 & 0 & -0.26477 \\
0 & 0 & 0 & 0 & 0
\end{bmatrix}.
\end{array}\vspace{1cm}$

It is interesting to note that each projection above is symmetric about its \textit{anti-diagonal}, the diagonal from the $(n,1)-$entry to the $(1,n)-$entry. This symmetry is in fact, always present in the optimal projection $Q=(q_{ij})$ from Theorem~\ref{summarizing thm} obtained by taking $z_i=1$ for all $i$. To see this, first observe that the function $h_t$ from equation~(\ref{ht equation}) satisfies the identity
$$\begin{array}{rl}
h_t(x)+h_t^{-1}(1-x)=1, & x\in[0,1].
\end{array}\smallskip$$
From here we have that 
$a_1+a_{n-1}=h_t(0)+h_t^{-1}(1)=1$, and by induction,
$$a_k+a_{n-k}=h_t(a_{k-1})+h_t^{-1}(a_{n-k+1})=h_t(a_{k-1})+h_t^{-1}(1-a_{k-1})=1\smallskip$$ for all $k\in\{1,2,\ldots, n\}$. Consequently, 
$$
\begin{array}{rcl}
q_{kk}&=&1-(a_k-a_{k-1})\vspace{0.1cm}\\
&=&a_{n-k}+a_{k-1}\vspace{0.1cm}\\
&=&a_{n-k}+(1-a_{n-k+1})\vspace{0.1cm}\\
&=&1-(a_{n-k+1}-a_{n-k})=q_{n-k+1,n-k+1}
\end{array}\smallskip
$$

 \noindent for all $k$. We now turn to the identity $q_{ij}=-\sqrt{(1-q_{ii})(1-q_{jj})}$ to conclude that that $q_{ij}=q_{n-j+1,n-i+1}$ for all $i$ and $j$, which is exactly the statement that $Q$ is symmetric about its anti-diagonal. An analogous argument using the formulas from \cite{MacDonaldProjections} demonstrates a similar phenomenon for optimal projections of rank $1$.
 
 \section{Conclusion}

The distance $\nu_{r,n}$ from the set of projections in $\mm_n(\cc)$ of rank $r$ to the set of nilpotent operators on $\cc^n$, as well as the corresponding closest projection-nilpotent pairs, are now well understood when $r=1$ or $r=n-1$. Of course, it is natural to wonder about the value of $\nu_{r,n}$ for $r$ strictly between $1$ and $n-1$. 

The difficulty in extending the above arguments to projections $P$ of intermediate ranks lies in deriving closed-form expressions for $\|E_{i-1}^\perp PE_i\|$. Computing these norms for projections of rank $r=1$ or $r=n-1$ was made possible by the simple structure afforded by such projections. In particular, we made frequent use of equations~(\ref{projection from sequence}) and (\ref{projection structure equation}) throughout the proofs of Lemma~\ref{Finding the entries and eigenvalues of Q_k^*Q_k} and Theorem~\ref{function definition of Arveson seminorm thm} to relate the off-diagonal entries of such projections to those on the diagonal. Analogous equations for projections of intermediate ranks become significantly more complex.

For small values of $r$ and $n$, the mathematical programming software \textit{Maple} was used to construct examples of rank $r$ projections $P_{r,n}$ in $\mm_n(\cc)$ which we believe are of minimal distance to $\mc{T}_n$. To ease the computations, the program was tasked with minimizing the maximum norm $\|E_{i-1}^\perp PE_i\|$ over all projections $P$ of rank $r$ with real entries and symmetry about the anti-diagonal. While it may not always be possible for such conditions to be met by an optimal projection of rank $r$, the computations that follow may still shed light on a potential formula for $\nu_{r,n}$.

The smallest value of $n$ for which $\mc{P}(\cc^n)$ contains projections of intermediate ranks is $n=4$. In this case, the intermediate-rank projections are those of rank $2$. We found that
$$P_{2,4}=\begin{bmatrix}
1/2 & 1/2 & 0 & 0\\
1/2 & 1/2 & 0 & 0\\
0 & 0 & 1/2 & 1/2\\
0 & 0 & 1/2 & 1/2
\end{bmatrix}$$
is an optimal projection of rank $2$ satisfying the conditions above. It is easy to see that $$\begin{array}{rl}\|E_{i-1}^\perp P_{2,4}E_i\|=1/\sqrt{2}=\nu_{1,2} & \text{for all}\,\,\,i,\end{array}\smallskip $$and hence $P_{2,4}$ is a direct sum of optimal rank-one projections in $\mm_2(\cc)$.

In $\mm_5(\cc)$, the intermediate-rank projections are those of rank $r=2$ or $r=3$. For such $r$, we obtained
$$\begin{array}{cc}P_{2,5}=\begin{bmatrix}
\phantom{-}0.42602 & -0.07632 & \phantom{-}0.22568 & \phantom{-}0.42334 & -0.09248\\
-0.07632 & \phantom{-}0.42127 & \phantom{-}0.23481 & -0.06022 & \phantom{-}0.42334 \\
\phantom{-}0.22568 & \phantom{-}0.23481 & \phantom{-}0.30541 & \phantom{-}0.23481 & \phantom{-}0.22568 \\
\phantom{-}0.42334 & -0.06022 & \phantom{-}0.23481 & \phantom{-}0.42127 & -0.07632 \\
-0.09248 & \phantom{-}0.42334 & \phantom{-}0.22568 & -0.07632 & \phantom{-}0.42602
\end{bmatrix}\phantom{.} & \text{and}\vspace{0.3cm}\\P_{3,5}=\begin{bmatrix}
\phantom{-}0.58296 & -0.29271  & -0.10684 & \phantom{-}0.12213 & \phantom{-}0.36209 \\
-0.29271& \phantom{-}0.62479 & -0.33169 & -0.15433 & \phantom{-}0.12213\\
-0.10684 & -0.33169 & \phantom{-}0.58448 & -0.33169 & -0.10684\\
\phantom{-}0.12213 & -0.15433 & -0.33169 & \phantom{-}0.62479 &-0.29271\\
\phantom{-}0.36209 & \phantom{-}0.12213 & -0.10684 & -0.29271 & \phantom{-}0.58296
\end{bmatrix},\end{array}\smallskip$$
with entries rounded to the fifth decimal place. Again, the norms $\|E_{i-1}^\perp P_{r,n}E_i\|$ share a common value, with $$\begin{array}{rl}
\displaystyle{\|E_{i-1}^\perp P_{2,5}E_i\|\approx 0.65270\approx\frac{1}{2}\sec\left(\frac{\pi}{\frac{5}{2}+2}\right)} & \text{for all}\,\,i,\,\,\,\,\text{and}\vspace{0.2cm}\\
\displaystyle{\|E_{i-1}^\perp P_{3,5}E_i\|\approx 0.76352\approx\frac{1}{2}\sec\left(\frac{\pi}{\frac{5}{3}+2}\right)} & \text{for all}\,\,i.
\end{array}\smallskip$$

In light of these findings, as well as the distance formulas that exist for projections of rank $1$ or $n-1$, we propose the following generalized distance formula for  projections of arbitrary rank. 

\begin{conj}\label{rank $k$ proj conj}
	For every $n\in\nn$ and each $r\in\{1,2,\ldots, n\}$, the distance from the set of projections in $\mm_n(\cc)$ of rank $r$ to $\mc{N}(\cc^n)$ is
		$$\nu_{r,n}=\frac{1}{2}\sec\left(\frac{\pi}{\frac{n}{r}+2}\right).$$
\end{conj}

Using a random walk process implemented by the computer algebra system \textit{PARI/GP}, we estimated the values of $\nu_{r,n}$ for all $r\leq n\leq 10$ without the additional assumptions described above. We observed only minute differences between these estimates and the expression from Conjecture~\ref{rank $k$ proj conj}. In many cases, these quantities differed by no more than $1\times 10^{-3}$.

The proposed formula from Conjecture~\ref{rank $k$ proj conj} merits several interesting consequences. Firstly, this formula suggests that $\nu_{r,n}=\nu_{kr,kn}$ for every positive integer $k$, meaning that a closest projection of rank $kr$ to $\mc{T}_{kn}$ could be obtained as a direct sum of $k$ closest projections of rank $r$ to $\mc{T}_n$. 
%From this it would follow that for some $n$ and $r\neq 1,n-1$, there exist optimal projection-nilpotent pairs that are \textit{not} diagonally unitarily equivalent. Indeed, the formula $\nu_{r,n}=\nu_{kr,kn}$ would imply that the pairs $(P_1,T_1)$ and $(P_2,T_2)$ given by
%$$\begin{array}{ccc}
%P_1=\begin{bmatrix}
%1/2 & 1/2 & 0 & 0\\
%1/2 & 1/2 & 0 & 0\\
%0 & 0 & 1/2 & 1/2\\
%0 & 0 & 1/2 & 1/2
%\end{bmatrix}, & T_1=\begin{bmatrix}
%0 & 1/2 & 0 & 0\\
%0 & 0 & 0 & 0\\
%0 & 0 & 0 & 1/2\\
%0 & 0 & 0 & 0
%\end{bmatrix} & \text{and}\vspace{0.3cm}\\
%P_2=\begin{bmatrix}
%1/2 & 0 & 1/2 & 0\\
%0 & 1/2 & 0 & 1/2\\
%1/2 & 0 & 1/2 & 0\\
%0 & 1/2 & 0 & 1/2
%\end{bmatrix}, & T_2=\begin{bmatrix}
%0 & 0 & 1/2 & 0\\
%0 & 0 & 0 & 1/2\\
%0 & 0 & 0 & 0\\
%0 & 0 & 0 & 0
%\end{bmatrix}\end{array}\smallskip$$
%are optimal for $n=4$ and $r=2$, yet there clearly does not exist a diagonal unitary $U$ such that $P_1=U^*P_2U$. Compare this observation with \cite[Theorem~6]{MacDonaldProjections} in the case that $r=1$, and Theorem~\ref{summarizing thm} in the case that $r=n-1$.
Notice as well that if the equation $\nu_{r,n}=\nu_{kr,kn}$ were true, it would follow that
$\nu_{1,n}=\nu_{r,rn}\leq \nu_{r,n}$
for each $n$ and $r$. Thus, a proof of Conjecture~\ref{rank $k$ proj conj}---or of the formula $\nu_{r,n}=\nu_{kr,kn}$---would validate Conjecture~\ref{MacDonald's Conjecture}.

\section*{Acknowledgements}
The author would like to thank Paul Skoufranis for many stimulating conversations and Boyu Li for providing the \textit{PARI/GP} script used to estimate $\nu_{r,n}$.

\vspace{0.2cm}

	\bibliography{DPNbib}

\begin{thebibliography}{1}

\bibitem{Arveson}
W.B. Arveson.
\newblock Interpolation problems in nest algebras.
\newblock {\em Journal of Functional Analysis}, 20(3):208--233, 1975.

\bibitem{Bini}
D.~Bini, Y.~Eidelman, L.~Gemignani, and I.~Gohberg.
\newblock The unitary completion and {QR} iterations for a class of structured
  matrices.
\newblock {\em Mathematics of Computation}, 77(261):353--378, 2008.

\bibitem{Hedlund}
J.H. Hedlund.
\newblock Limits of nilpotent and quasinilpotent operators.
\newblock {\em The Michigan Mathematical Journal}, 19(3):249--255, 1972.

\bibitem{Her1}
D.A. Herrero.
\newblock Normal limits of nilpotent operators.
\newblock {\em Indiana University Mathematics Journal}, 23(12):1097--1108,
  1974.

\bibitem{HerreroUnitaryOrbitsofPower}
D.A. Herrero.
\newblock Unitary orbits of power partial isometries and approximation by
  block-diagonal nilpotents.
\newblock In {\em Topics in modern operator theory}, pages 171--210. Springer,
  1981.

\bibitem{MacDonaldProjections}
G.W. MacDonald.
\newblock Distance from projections to nilpotents.
\newblock {\em Canadian Journal of Mathematics}, 47(4):841--851, 1995.

\bibitem{MacDonaldIdempotents}
G.W. MacDonald.
\newblock Distance from idempotents to nilpotents.
\newblock {\em Canadian Journal of Mathematics}, 59(3):638--657, 2007.

\bibitem{PowerDistance}
S.C. Power.
\newblock The distance to upper triangular operators.
\newblock In {\em Mathematical Proceedings of the Cambridge Philosophical
  Society}, volume~88, pages 327--329. Cambridge University Press, 1980.

\bibitem{SalinasDistance}
N.~Salinas.
\newblock On the distance to the set of compact perturbations of nilpotent
  operators.
\newblock {\em Journal of Operator Theory}, pages 179--194, 1980.

\end{thebibliography}
	\bibliographystyle{plain}
	
	\medskip
	
	\Addresses

\end{document}